\documentclass[12pt]{article}
\usepackage{epsfig,psfrag,amsmath,amssymb,latexsym,bbm}
\usepackage{amscd }
\usepackage{amsfonts}
\usepackage{graphicx}
\newtheorem{theorem}{Theorem}[section]

\newtheorem{lemma}{Lemma}[section]

\newenvironment{proof}[1][Proof]{\par\noindent\textbf{#1.} }{\hfill~\rule{0.5em}{0.5em}\medskip}
\newcommand{\N}{\mathbb{N}}

\newcommand{\bbr}{\mathbb{R}}

\newcommand{\E}{\mathbb{E}}

\newcommand{\call}{\cal{L}}
\def\calr{{\cal R}}
\def\cala{{\cal A}}
\renewcommand{\P}{\mathbb{P}}
\def\l{\ell}
\numberwithin{equation}{section}

\def\var{\mathop{\rm Var}}
\let\epsilon=\varepsilon
\def\ep{\varepsilon}
\begin{document}
\title{Sparse long blocks and the variance of the length\\
 of longest common subsequences in random words}
\author{S.~Amsalu\thanks{Department of Information Science, 
Faculty of Informatics,
University of Addis Ababa, Ethiopia}
 \and C.~Houdr\'e\thanks{School of Mathematics, Georgia
 Institute of Technology, Atlanta, GA 30332
\hfill\break {\tt houdre@math.gatech.edu}. Supported in part by
Simons Foundation grant 246283 and by a Simons Fellowship grant 267336}
\and H.~Matzinger\thanks{School of Mathematics, Georgia
 Institute of Technology, Atlanta, GA 30332
\hfill\break {\tt matzi@math.gatech.edu}. Supported in part by
Simons Foundation grant 317900} }

\maketitle

\begin{abstract}
Consider two independent random  strings having  same length and
 taking values uniformly 
in a common finite alphabet.  We study the order of the variance 
of the length of the longest common subsequences (LCS) of these strings when     
long blocks, or other types of atypical substrings,
 are sparsely added into one of them.
Under weak conditions on the derivative
 of the mean LCS-curve, 
the order of the variance of the LCS is shown to be linear  
in the length of the strings.  
We also argue that
our proofs carry over to many models used
by computational biologists to simulate DNA-sequences.
This is the first result where
the open question of the order of the fluctuation
of the LCS of random strings is solved for a realistic 
model. Until now, this type of result 
had only been established for low entropy cases.
\end{abstract}

\section{Introduction} 

Let $x$ and $y$ be two finite strings.  A common subsequence 
of $x$ and $y$ is a subsequence which is a subsequence 
of both $x$ and $y$, while a longest common subsequence (LCS)  
is a common subsequence of maximal length.
  Common subsequences 
can be represented via alignments, and for this 
the letters which are part of the subsequence get aligned
with identical letters, while the remaining
letters get aligned with gaps.

Let us give an example of common subsequences
and alignments with gaps:
 take $x=heinrich$ and let $y=enerico$. 
Then $z=ni$ is a common subsequence of $x$ and $y$,
indicating that the string $ni$ can be obtained from both $x$ and $y$
by just deleting letters.
The subsequence $ni$, corresponds to the alignment
with gaps:
$$
\begin{array}{c|c|c|c|c| c|c|c|c|c| c|c|c|c|c}
x& &h& &e&i&n& &r& &i&c& & &h\\\hline
y& & &e& & &n&e& &r&i& &c&o& 
\end{array}
$$
The common subsequence $ni$ is not 
of maximal length, the LCS 
is $enric$, and 
the corresponding alignment is given by:
$$
\begin{array}{c|c|c|c|c|c|c|c|c|c|c|c}
x& &h&e&i&n& &r&i&c& &h\\\hline
y& & &e& &n&e&r&i&c&o& 
\end{array}
$$
Often, a long LCS indicates that the strings are related.  
In this article, we only consider alignments 
which align same letter pairs,  every such 
alignment defines a common subsequence, and 
the length of the subsequence corresponding 
to an alignment is called the {\it score of the alignment}.  
The alignment representing a LCS 
is also called an {\it optimal alignment} (OA). 
In the above example, the length of the LCS is five which is denoted
by:
$$|LCS(heinrich;enerico)|=|LCS(x,y)|=5.$$

Longest Common Subsequences (LCS) and Optimal Alignments (OA)
are important tools used 
for string matching
in Computational Biology and Computational Linguistics
\cite{Capocelli1,watermanintrocompbio,Watermangeneralintro}.
A main application is to the automatic recognition of related DNA pieces. 
In that context, 
it is anticipated that if two DNA-strings have a common ancestor, 
then they will have a long LCS.  Could 
it be that, unrelated (independent) 
strings have nonetheless a long LCS?  How likely is such an event?  
This, of course, depends on the probabilistic model generating the strings. 
To answer the previous questions, the behavior, for $n$ large, of both the 
expectation $\E LC_n$ and 
the variance $\var LC_n$ need to be understood.
Throughout $LC_n$ is the length of the LCS
of the random strings $X= X_1\cdots X_n$ and $Y= Y_1\cdots Y_n$ so that
$$LC_n:=|LCS(X_1X_2\cdots X_n;Y_1Y_2\cdots Y_n)|.$$

The asymptotic behavior of the expectation and the variance of the length 
of the LCS of two independent random strings has been studied by probabilists, physicists,
computer scientists and computational biologists.
It can also be formulated as a last passage percolation problem with dependent weights. 
The problem of finding the fluctuation order for
first and last passage percolation
has been open for a several decades. There 
has been, however, a well-known
breakthrough for a related problem,
that is for the  Longest Increasing
Subsequence (LIS) of a random permutation
\cite{BaikDeiftJohansson99}. 
For the LIS of a uniform random permutation of $\{1,2,\dots, n\}$,
the order of the fluctuation is
the cubic root of the expectation and not its square root
and this is also true for the LCS of two uniform random 
permutations of $\{1,2,\dots, n\}$ (see \cite{K,J,HI}).
For the LIS of a random word,
both the expectation and the variance are linear \cite{ITW1,ITW2,TW}.  
For the LCS of random words, the expectation is of order $n$, and so
if the fluctuations were also
a cubic root of the expectation,  then
$\var LC_n$ should
be of order $n^{2/3}$.
This is the order of magnitude conjectured in
\cite{Sankoff1} 
for which several heuristic proofs have been claimed.  
This conjectured order might even seem more plausible in view of the 
recent solution 
\cite{MMN,MN} to the Bernoulli
matching problem where the variance is shown to be of order $n^{2/3}$.
Although this cubic-root behavior is the correct order for the LCS of two uniform random
permutations of $\{1,\dots, n\}$,
we believe this order to be incorrect for the LCS of random words.
The cubic-root  claims might happen because 
the LCS length can be viewed as a last passage percolation problem 
with dependent weights,  and so, for short sequences 
the dependence in the weights does not have a 
strong influence, the LCS then behaves as if the 
weights were independent.  
However, for short sequences this order might be what one approximately 
observes in simulations.  For the LCS of independent
iid strings the order of magnitude
of the variance is, in general, not known. (Except
for various cases, see
\cite{increasinglcs}, \cite{fluctlcsnotsym}, \cite{VARTheta}, for which 
the variance is asymptotically linear in the length 
of the strings
considered.) 
Below we determine 
the correct asymptotic order of the variance for the LCS
of uniform iid sequences ``with artificially added
impurities" provided any of the condition given in
Subsection \ref{conditions} below hold. These conditions
 are not directly needed for the current
proofs, rather, they are needed for the results obtained
in \cite{amsaluhoudrematzi2} 
which in turn are
used in the present article.

A subadditivity argument  in
\cite{Sankoff1} 
shows  the existence of the limit: 
\begin{equation}\label{eq:sub}
\gamma^*_k:=\lim_{n\rightarrow\infty}\frac{\E\, LC_n}{n},
\end{equation}
where $X$ and $Y$ are two stationary ergodic strings independent of each
other and where the
constant $\gamma^*_k>0$ depends on the distribution
of $X$ and $Y$ and on the size $k$ of the alphabet.
Even for the simplest distributions, such as iid strings
with binary equiprobable letters, the exact value
of $\gamma^*_k$  is unknown, and extensive simulations have been performed
to obtain approximate values \cite{Baeza1999,B,Paterson1,BdM1,BdM2,Deken}.

The speed of convergence to the expected length in \eqref{eq:sub} was further
determined in \cite{Alexander,alexander2}, 
showing that for iid sequences,  
\begin{equation}\label{iidseq}
\gamma^*_k n-C\sqrt{n\log n}\le \E\, LC_n\le\gamma^*_k n,
\end{equation}
where $C>0$ is a constant depending neither on 
$n$ nor on the distribution of $X_1$.

As already mentioned, there exist contradicting
conjectures for the order  of the variance
of the LCS. Our present result (Theorem~\ref{theoremfluct})
establishes the order conjectured 
 in \cite{Waterman-estimation}
for an iid distribution. We prove it, however,
for iid sequences with added impurities (sparse long blocks or
atypical substrings),
assuming also some differentiability conditions on the
mean LCS-cure $\gamma_k$. This list of conditions,
any of them making our main result hold, is
 given in Subsection \ref{conditions}.
They are not used directly in the current proofs,
but rather to make some theorems hold which are given
in \cite{amsaluhoudrematzi2}. 
 The mean LCS-curve is the rescaled 
 expectation of the LCS when the two sequences are taken to be of
different length but in a fixed proportion. (See \eqref{meancurve}.) 
The impurities or ``long blocks" as we call them,
are substrings consisting only of one symbol which can be different
from block to block.
For that model, the variance is shown to be of order 
 $\Theta(n)$, i.e.,
there exist two constants $C_2>C_1>0$ independent of $n$,
such that $C_1n\leq \var LC_n\leq C_2n$, for all
natural number $n$. (Here $LC_n$ is the length
of the LCS of the two independent sequences $X$ and $Y$ of length $n$,
one of the two sequences having sparsely added long 
blocks.)
It is rather interesting 
that the mere differentiability of 
the mean curve at its maximum,  implies
a certain order of magnitude for the variance.
Note that 
\cite{Steele82,Steele86} proved that 
$\var LC_n\leq n$, and so only good lower 
bounds for the variance of $LC_n$ are needed.  (Simulation studies are not
that numerous in case of the variance and at times contradict each other.)
 
Still for iid 
sequences with $k$ equiprobable letters,
the order of $\var LC_n$ remains
unknown. We hope nonetheless, that similar ideas 
could be helpful in fully tackling this problem. 

\centerline {\bf Overview of the main result of this
paper}

We first need a few definitions:
Let $V_1,V_2,\ldots$ and $W_1,W_2,\ldots$ be two independent iid
sequences with $k$ equiprobable letters, and let
$$\gamma^*_k:=\lim_{n\rightarrow\infty}
\frac{\E|LCS(V_1V_2\cdots V_n; W_1W_2\cdots W_n)|}{n}.$$
As already mentioned, the exact value of 
$\gamma^*_k$ is unknown, but lower and upper
bounds are available, e.g.,
\begin{equation}
\label{table}\begin{array}{cc|c|c|c|c}
k       & &2   &3&4&\cdots \\\hline
\gamma^*_k& &0.812 &0.717 &0.654 &\cdots 
\end{array}
\end{equation}
where the precision in the above table is about $\pm 0.01$.
The expected length of the LCS
of two independent iid sequences both of length $n$ is thus about 
$\gamma^*_kn$,
up to an error term of order not more than a
 constant times $\sqrt{n\log n}$ (see \eqref{iidseq}). 

We can also consider two sequences of different lengths,
but in such a way that the two lengths are in a fixed proportion of each
other.
To do so, let 
\begin{equation}
\label{gammaknp}\gamma_k(n,q):=
\frac{\E|LCS(V_1V_2\cdots V_{n-nq}; W_1W_2\cdots W_{n+nq})|}{n},
\end{equation}
where   $q\in[-1,1]$,  and let 
\begin{equation}
\label{meancurve}
\gamma_k(q):=
\lim_{n\rightarrow\infty}\gamma_k(n,q),
\end{equation}
which again exists by subadditivity arguments.
 The function $q\mapsto \gamma_k(q)$ is called the {\it mean LCS-function},
it is symmetric around $q=0$ and concave
and it thus has a maximum at $q=0$ which is equal to $\gamma^*_k$
(see \cite{amsaluhoudrematzi2}).
This function  corresponds to the 
wet-region-shape in first passage percolation.

The main result of this paper is the variance result
  given in Theorem \ref{theoremfluct}. It states  that
  provided certain differentiability conditions
on the mean LCS-curve,
$$\var LC_n=\Theta(n),$$
for a model  with sparse long blocks inserted into an iid sequence
with $k$-equiprobable letters. 
(A  {\it block} is a maximal contiguous
substring consisting of only one symbol.)
The conditions used are given in Subsection \ref{conditions}.

The model considered will be described in detail at
the beginning of Section~\ref{severallong} but 
let us, nevertheless,  already give 
an overview of it. Let $\beta$ and $p$ to be any real constants 
(independent of $n$ and $d$) such that
$$\frac12<\beta<1$$
and
$$0<p<1.$$
Basically, the strings 
$$X=X_1X_2\cdots X_n$$
and
$$Y=Y_1Y_2\cdots Y_n$$
are iid strings except that in some sparse locations
in $X$ we insert long blocks. For the rest $X$  and $Y$
are independent of each other with $k$-equiprobably 
letters. The possible locations in $X$ for the long blocks
are, say, $d,3d,5d,\ldots,2dm-d$, where again $n=2dm$.
For each possible location throw
independently the same (possibly biased) coin to decide whether or not to place 
a long block there, and the probability to place in a given location
a long block is $p$. Placing a long block means that 
a iid piece of equal length is replaced by the long
block. The long blocks have length about
$\l=d^\beta$, so they are much smaller in size then the intervals
$[2id,2(i+1)d]$ which ``each of them hosting at most one such long block".
We take $d$ large, but fixed,
while $n$ goes to infinity.

Let us present an example. Take $d=5$
and $\l=4$, while $m=2$. The length of the sequences
$X$ and $Y$ is thus $n=2dm=20$.
Consider binary sequences
so that $k=2$. We are thus throwing an unbiased coin
independently $20$ times to obtain the sequence $Y$.
For example we could have:
$$Y=00101 11010 00110 10101.$$
Then, we throw our unbiased coin again $20$ times 
to obtain the sequence $X^*$.
The sequence $X^*$ is thus also uniform iid, and  we proceed 
 to add long blocks into $X^*$ in order to obtain 
the string $X$. The potential
places for long blocks are the integer intervals 
$[d(2i-1)-\l/2,d(2i-1)+\l/2]$, where $i=1,2,\ldots,m$.
 In the present example,
there are only $m=2$ such intervals:
\begin{equation}
\label{intervals37}
[3,7]\;\;{\rm and}\;\; [13,17].
\end{equation}
Assume that after having thrown our unbiased coin $n=20$ times,
we obtained for $X^*$ the sequence
$$X^*=01{\bf 010 10}001 01{\bf 001 01}110,$$
where the bold face substrings could
get replaced by long blocks.
The next step is to throw a (possibly
biased)  coin  for each of the intervals which could
get a long block. (We will thus throw the coin $m=2$ times.)
In this way, we decide for each of the
intervals in \eqref{intervals37} whether or not
 there will be a long block covering it.
If the corresponding Bernoulli random variable $Z_i$ is equal to $1$,
then there will
be a long block covering the interval $[d(2i-1)-\l/2,d(2i-1)+\l/2]$ and if
it is equal to
0 then otherwise.
The probability of a long block is thus equal to  $\P(Z_i=1)=p$.
In the present example assume we throw our coin twice 
and obtain $Z_1=1$ and $Z_2=0$, then 
in the interval $[3,7]$ we place a long block, say made up of zeros, 
while in $[13,17]$
we leave things
as they are. With these modifications we obtain:
$$X=01{\bf000 00}001 01{\bf 001 01}110.$$
In other words, to obtain $X$ from $X^*$, we 
simply fill each integer interval $$[d(2i-1)-\l/2,d(2i-1)+\l/2]$$
for which $Z_i=1$, with all the same bits and leave everything else 
unchanged.  In our example,
$$
\begin{array}{c|c|c|c|c| c|c|c|c|c| c|c|c|c|c| c|c|c|c|c| c|c}
X^*& &0&1&{\bf0}&{\bf1}&{\bf 0}&{\bf 1}&{\bf 0}&0&0&1& 0&1&{\bf 0}&{\bf0}&
{\bf1}&{\bf 0}&{\bf1}&1&1&0
\\\hline
X& &0&1&{\bf0}&{\bf0}&{\bf 0}&{\bf 0}&{\bf 0}&0&0&1& 0&1&{\bf 0}&{\bf0}&
{\bf1}&{\bf 0}&{\bf1}&1&1&0 \\\hline 
\end{array}\,.
$$
where the bold digits are the places of (potential) long blocks.

As already mentioned, the 
main result of the present paper (Theorem~\ref{theoremfluct}) 
is that
the order of magnitude of $\var LC_n$ is linear, i.e.,
$$\var LC_n=\Theta(n),$$
for our model of inserted long blocks provided any of the conditions
of Subsection \ref{conditions} hold.

To prove this result,
 we take the order of the 
length of the inserted long blocks larger than $\sqrt{d}$,
but smaller than $d$.
More precisely, 
we take $d$ and $p\in(0,1)$ fixed,
while $n$ the common length of $X$ and $Y$ goes to infinity.
We take a parameter $\beta$ 
not depending on $d$ such that $1/2<\beta<1$, and set
the length of the long blocks to be $\l=d^\beta$.
Our result holds,  for all $d$ large
enough, but fixed, and assuming  a block-length of $\l=d^\beta$ 
($d$ does not need to be very
large for our fluctuation result to hold).
Note also that the length of the long block
does not need to be exactly $\l$, it could be a little
bigger, but this is of no real importance in the present investigation.

\medskip
\centerline{\bf Main idea behind the proof}

In \cite{amsaluhoudrematzi2}, the situation where both sequences have length 
only $2d$ and where 
only one long block  is inserted into the middle of one of the strings 
is studied. 
Again, the long block has length
$d^\beta$ and replaces an iid part of equal length.
It is then  shown in \cite{amsaluhoudrematzi2} that with high probability 
the effect of replacing 
 the long block by an iid part of equal length tends
to increase  the LCS by a quantity proportional to the length
of the long block. 
In the current paper, speaking loosely
 we have the concatenation of many times
the situation described in \cite{amsaluhoudrematzi2}.
  
Now, in \cite{VARTheta}, \cite{BM}
a coupling technique is used to show that the variance
of the LCS  of random strings is proportional
to their length. That technique  allows to show the linear order
for the variance as soon as one is able to prove that a random change
to the strings has a biased effect on the scores. The random
change cannot be operated in one predetermined location
of the strings, but must be operated in a random location which macroscopically
is equally likely to occur in any part of the string.
This technique is used, adapted to our current long block
situation. The random
change we consider in this paper is defined as follows:
we take a long block at random among all long blocks
and reverse it to iid. If we can show that this random change
 has a biased effect 
on the LCS then 
one gets the order 
\begin{equation}
\label{orderlcn}
\var LC_n=\Theta(n).
\end{equation}
This equivalence between the order \eqref{orderlcn}
and the biased effect of a random long block replacement
is stated and proved in Theorem \ref{equivalence}.

As already mentioned,
for a one-long-block-only-situation the biased effect of changing the long block to iid
has already been obtained in \cite{amsaluhoudrematzi2}. 
So, {\it the main thing to prove in the current article is that:
 the biased effect in the one-long block situation with strings of length
$2d$ 
implies a biased effect  in our multiple long block model,
of length ``many times $2d$," studied
in the current article.} Again, in the current paper the long block 
which is turned into iid is chosen at random 
among all long blocks. We only need to prove an expected biased effect
when we chose the long block at random. This implies that
 if a small number of the long blocks when changed to iid do not produce
the desired biased effect it does not matter as long as the majority of
them does! 

\subsection{Conditions for the validity of main result}
\label{conditions}
In order to hold, the main result of the present paper requires
some differentiability condition on $\gamma_k$.
 These conditions are not directly needed in the current proofs,
but they are needed to make the results
of \cite{amsaluhoudrematzi2}, which are needed below, hold.

First note that, since it is concave, $\gamma_k$ 
has non-increasing left and a right derivatives at any $p\in (-1,1)$, 
with $\gamma_k^\prime(p^-) \ge \gamma_k^\prime(p^+)$, while by symmetry, 
$\gamma_k^\prime(p^\pm)=-\gamma_k^\prime((-p)^\mp)$.   

Next, let $0\le p_M < 1$ be 
the largest real for which  
$\gamma_k$ is maximal. Hence, $[-p_M,p_M]$  
is the largest interval on which $\gamma_k$ is everywhere 
equal to its maximal value $\gamma_k(0)$, i.e.,  
$[-p_M,p_M]=\gamma^{-1}(\{\gamma_k(0)\})$.

Our theorems will be verified under any one 
of the following four conditions:  

\begin{enumerate}
\item{} The mean LCS-function $\gamma_k$ is strictly concave in a 
neighborhood of the origin and is 
differentiable at $0$ (and so $p_M=0$ and $\gamma_k^\prime(0) = 0$).  
\item{}The function $\gamma_k$ is differentiable at $p_M$, i.e., $\gamma_k^\prime(p_M^+) = 
\gamma_k^\prime(p_M^-)$ 
and therefore (either by symmetry or since $\gamma_k^\prime(p_M^-) = 0$ 
if $p_M>0$ ) $\gamma_k^\prime(p_M) = 0$.  
\item{} The absolute value of $\gamma_k^\prime(p_M^+)\le 0$ is dominated 
by the absolute value of $\gamma_k(0)-(2/k)$:    
$$\left|\frac{\gamma_k^\prime(p_M^+)}{2}\right|<\left|\frac{\gamma_k(0)}{2}-\frac{1}{k}\right|.
$$
\item{} The function $\gamma_k$ is strictly concave in a neighborhood of the origin 
and its right derivative at the origin is such that:   
$$\left|\frac{\gamma_k^\prime(0^+)}{2}\right|<\left|\frac{\gamma_k(0)}{2}-\frac{1}{k}\right|.
$$

\end{enumerate}

Clearly, $1 \implies 2\implies 3, 1\implies 4$.  

In \cite{amsaluhoudrematzi2}, the main results are proved under
the assumption of Condition 2.  But, it is 
also explained in the summary of the proof (Section 2 there)
how Condition 3 and Condition 4 would work. With 
Condition~3, the notations for the proofs in \cite{amsaluhoudrematzi2} 
would become cumbersome with an additional term appearing everywhere. 
So, in \cite{amsaluhoudrematzi2} it was decided for the formal
proof to stick to Condition 2.   From our simulations,
we have no doubt that even Condition 1 holds.
It should also be noted that Condition
3, unlike the others, can be verified up
to a certain confidence level by Montecarlo 
simulations. This makes this condition rather important.

\subsection{Motivation from biology and possible extensions}

The present paper was partly motivated by remarks from
computational biologists to the effect that DNA distribution
is not homogeneous. Rather there are different parts,
with different biological functions (exon,
coding parts, non-coding parts, $\dots$). These different parts,
 having different
lengths and each having its own distribution,
 are often modeled by computational biologists
using hidden Markov chains; the hidden states determining
the parts. Once the hidden states are determined, 
the DNA-sequence is drawn, by using the corresponding
distribution for each part.

The reader might wonder how realistic our present long block model is,
in view of this hidden-Markov model.
Why did we add long blocks in predetermined
positions and why do they only get added 
into one sequence and not both? Also, in DNA-sequences
there are typically no long blocks.
Let us present the various restrictions of our model
and explain 
which features are present only to simplify the, already involved, notation,
but do not represent a fundamental restriction:

\begin{enumerate}
\item{}The first restriction is that we add long blocks in
predetermined locations. This restriction is only there to simplify notation.
The same proof works if we use a Poisson point process
with intensity-parameter $\lambda=1/2d$ to determine the locations
of the long blocks. Also, we could take the length of the long blocks
to be a geometric random variable with expectation $\l=d^\beta$.
\item{}Another quite unnatural restriction is to put
the long blocks only into one sequence. This is done again to simplify
notations. If the starting location of the long blocks is given by a Poisson
point process with intensity $\lambda=1/2d$, then we can
add long blocks in both sequences $X$ and $Y$. We would then
use independent Poisson point processes 
with the same intensity for both $X$ and $Y$. 
The proofs  presented here work as well for this case.
\item{}The model with long blocks added in Poisson locations,
is very similar to a $2$-state hidden Markov chain. For this we
could take the hidden states to be $L$ and $R$. The state
$L$ would correspond to a long block, while $R$ would
be the places where the string is iid. The transition probabilities
from $L$ to $R$ would be $1/d^\beta$, while from
$R$ to $L$ it would be $1/2d$. Again, 
for this hidden $2$-state model, proofs very similar to the ones presented
here will give the linear order of the variance,  
but the notations would have to become even more cumbersome.
\item{}In DNA-sequence, there are no long strings consisting
only of one symbol. So, the long block model may at first not
look very realistic. However, in place of long blocks, we can take
 pieces generated by another distribution. For this we take
two ergodic distributions. Typically one could use a 
finite Markov chain or a hidden Markov model with finitely
many hidden states. For each different part, we would use the corresponding
distribution. 
We could use a hidden Markov model first to determine
which positions belong to which part and then fill the part
with strings obtained from the corresponding distribution.
(These corresponding distributions will again typically be hidden Markov
with finitely many hidden states or 
Markov or finite Markov, maybe even Gibbs.)
The hidden states could again be $L$ and $R$. (To simplify things
here we assume that there are only two DNA-parts.) But this time the
 state
$L$ would not correspond to a long block. Rather we would have two 
stationary, ergodic distributions 
$\mu_L$ and $\mu_R$. The places with hidden state $R$
would get the DNA-sequence drawn using $\mu_R$, while
for the positions with hidden state $L$, we would draw the 
DNA-sequence from $\mu_L$. The transition probabilities
between $L$ and $R$ would be as before:
from $L$ to $R$ it would be $1/d^\beta$, while from
$R$ to $L$ it would be $1/2d$.
We believe that our  current
approach to determine the order of the variance  could work
for this hidden Markov chain case, 
provided we had:
\begin{equation}
\label{M<R}
\gamma_{L,R}(q)<\gamma_R(q),
\end{equation}
for all $q$ in an appropriate closed interval around $0$.
Here, 
$\gamma_{L,R}(q)$ is the coefficient for the mixed model:
$$\gamma_{L,R}(q):=
\lim_{n\rightarrow\infty}
\frac{\E|LCS(V_1V_2\cdots V_{n-nq}; W_1W_2\cdots W_{n+nq})|}{n}\,, 
$$ when the string $V_1,V_2,\dots$ is drawn according to $\mu_L$ and
$W_1,W_2,\ldots$ is drawn according to $\mu_R$. Similarly,
$\gamma_R(q)$ is the parameter when
both sequences are drawn according to $\mu_R$:
 $$\gamma_R(q):=
\lim_{n\rightarrow\infty}
\frac{\E|LCS(U_1U_2\cdots U_{n-nq}; W_1W_2\cdots W_{n+nq})|}{n}, 
$$
where both sequences $U_1,U_2,\ldots$ and $W_1,W_2,\ldots$ are
drawn independently from each other with distribution
$\mu_R$. The condition \eqref{M<R} makes sense: it seems
clear that when aligning two sequences drawn from
the same distribution typically we should get a longer LCS
 than if we align sequences from a different
distribution. 
This  might be difficult to prove theoretically.
(This is the reason for considering long blocks, since they make
this kind of condition easily verifiable.)
Also, we would need for $\gamma_R$
to satisfy some differentiability property at all its maximal points.
In fact, instead of long blocks, any atypical long substrings such that 
its asymptotic expected LCS is smaller than $\gamma^*_k$ will do.
\item{}A true restriction of our method is that the long blocks are of order
greater than $\sqrt{d}$. Our current methodology does not carry over
when this is lacking. It should be noted however
that different parts (exon, coding, non-coding)
 of DNA are often pretty long,
so the current assumption might not be totally unrealistic.
We do not know how to treat the case of added long blocks
with length below $\sqrt{d}$.
\item{} We also assumed that the long blocks have length
of order below a constant times $d$. For long blocks-lengths
of order $d$ times a constant, a different paper would need to
be written. Clearly this seems within reach, considering
the present results.
\end{enumerate}
Summarizing: if we are willing
to accept the differentiability property of the function $\gamma_R$,
and that the condition \eqref{M<R} holds,
we probably should be able to get
$$\var LC_n=\Theta(n),$$
for a whole range of distributions
used  in practice to model DNA.
We plan to investigate this problem in the future.

Let us briefly describe the content of the rest of the paper.
In the next section,
the problem of finding the order of the variance is first reduced
to the biased effect of long blocks. For this,
we choose one long block at random and change it back to iid.
Theorem~\ref{equivalence} then states that if such a random alteration
has typically a sufficiently strong biased effect,
then the linear order of the variance follows.
After Theorem~\ref{equivalence}, the rest of the section
is  devoted to establishing the biased effect
of the long block replacement. 
Section \ref{gail} is dedicated to proving that the 
biased effect in the one long block
situation implies the biased effect in the multiple long block situation.
The Appendix then explains what adaptation of the one-long block situation
of \cite{amsaluhoudrematzi2}
is needed for the current paper.

\section{Long blocks and the variance}
\label{severallong}

In the present section, we consider strings of length $n$ 
with many long blocks
added. 
This many-long-blocks model was briefly explained
in the introduction and
it is precisely defined now:
First, the string $Y=Y_1\cdots Y_n$ is iid with $k$ equiprobable symbols.
Next, $d$ is taken large, but fixed
as $n$ goes to infinity, and we partition the iid sequence
$X^*=X^*_1\cdots X^*_n$ into pieces of length $2d$,
so that $n=2dm$. 
We then insert, say, 
in the middle of each of these pieces at most one long block,
deciding at random which pieces get a long block of length $\ell$ and which
do not. In all the places where there is no long block, the sequence
is iid with $k$ equiprobable symbols. 
Let us explain in more details how this string $X$ is defined:
For each $i=1,2,\dots, m$, let $J_i$ be the interval
$$J_i:=\left[(2i-1)d-\frac\l2,(2i-1)d+\frac\l2\right],$$
i.e., $J_i$ is the $i$-th place 
 where a long block 
could be introduced. 
We assume that
$X^*=X_1^*X_2^*X_3^*\cdots X_n^*$ is iid uniform.
The string $X$ is equal to $X^*$ everywhere except possibly at
the places where we put long blocks:
$$X_i:=X^*_i\,,\;\forall i\in [1,n]-\bigcup_{j=1}^m J_j.$$
Let $Z_i$ be the Bernoulli random variable which, when equal to one,
places a long block into the interval 
$J_i$. Hence,  
$$Z_i:=1\;\;{\rm implies}\;\;
X_{j_1}=X_{j_2}\;\;\forall j_1,j_2\in J_i,$$
and let
$Z_1,Z_2,\ldots,Z_m$ be iid Bernoulli random variables
with 
$$\P(Z_i=1)=p,$$
where $p\in(0,1)$. Hence, $p$ is nothing but the probability
to have a long block introduced artificially into
one of the possible locations. (The variables $Z_1$, $Z_2$,\dots,$Z_m$
are all independent of $X^*$ and $Y$, and the string $Y=Y_1Y_2\cdots Y_n$
is independent of $X$ and $X^*$.) Moreover the
 strings are drawn from an alphabet 
$\cala=\{\alpha_1,\alpha_2,\dots,\alpha_k\}$ with $k$
equiprobable symbols:
$$\P(X_i=\alpha_j)=\P(X_i^*=\alpha_j)=\P(Y_i=\alpha_j)=\frac{1}{k},$$
for all $i= 1,\dots, n=2dm$ and all $j\in\{1,2,\ldots,k\}$.

We are now ready to state the main result
of the paper. It gives 
the asymptotic  order of the variance
of the LCS of $X$ and $Y$ for the distribution
with many long blocks added. It is valid for any alphabet size $k$.

\begin{theorem} 
\label{theoremfluct}
Let $X$ and $Y$ be
two independent strings of length $n=2dm$ drawn
from an alphabet with $k$ equiprobable letters, $k\ge 2$.
Let $Y$ be iid and let $X$ be a string
with artificially long blocks randomly inserted into
some of the location $J_1,J_2,\ldots, J_m$
where $J_i=\left[(2(i-1)d-\frac\ell 2, 2(i-1)d+\frac\ell 2\right]$,
$i=1,\dots, m$ and $\ell=d^\beta$, $\frac12<\beta <1$.
Each of the locations has a probability $p$ to
receive a long block independently of the others.
Outside the long block areas, the string $X$ is
iid.
Let the concave function $\gamma_k$ 
be differentiable at its maxima,
then there exists $d_1$
such that for all $d\geq d_1$ independent of $n$, 
$$\var LC_n=\Theta(n).$$
\end{theorem}

For the above theorem to hold,  it is enough to show
that the change of one long block (picked at random)
 induces  an expected increase
in the LCS.  For this we choose in $X$ one of the long blocks at
random and change it back into iid. We assume that all the long
blocks have equal probability to get picked and the string
obtained by changing one long block into iid
is denoted by $\tilde{X}=\tilde{X}_1\tilde{X}_2\cdots \tilde{X}_n$.
Let us describe $\tilde{X}$ a little bit more formally:
First recall that $X^*$ denotes ``the iid string $X$ before
the long blocks are introduced".
Let $N$ be the total number of long blocks in $X$, i.e.,
$$N:=\sum_{i=1}^mZ_i,$$
and let $i(j)$ be the index of the $j$-th long block,
i.e., if 
$$\sum_{s=1}^{i-1}Z_s=j-1,\quad \sum_{s=1}^iZ_s=j,$$
then $i(j):=i$.

Next, let $M$ be a random variable which is uniform on $\{1,2,\ldots,r\}$
when $N=r$:
$$\P(M=j\mid N=r)=\frac{1}{r},\quad \forall j\leq r, $$
and assume that conditionally on $N=r$, the variable
$M$ is independent of $X$, $X^*$ and $Y$.
The block we change has index $i(M)$, therefore
$$\P(\tilde{X}_s=X_s,\forall\; s\notin J_{i(M)})=1$$
and
$$\P(\tilde{X}_s=X^*_s,\forall\; s\in J_{i(M)})=1.$$
In other words, the strings $X$ and $\tilde{X}$ are the same
everywhere except on the interval $J_{i(M)}$ and
on that interval, $\tilde{X}$ is equal to the iid
sequence $X^*$.
We are now ready to formulate the result stating that in order to show 
that $\var LC_n=\Theta(n)$, it is
enough to prove that 
the randomly changed block typically has a positive biased effect
on the length of the LCS: 

\begin{theorem}
\label{equivalence}
Let $d_0\in\N$, and let
there exist two constants $c_1,c_2>0$ independent of $n$ 
such that
\begin{equation}
\label{bias}
\P\left(\E\left.\left(|LCS(\tilde{X}; Y)|-|LCS(X;Y)|
\right|X,Y\right)\geq
c_1d^\beta_0\right)\geq 1-e^{-c_2n},
\end{equation}
for all $n$ large enough.
Then, taking $d_0=d$, it follows that
$$\var LC_n=\Theta(n).$$
\end{theorem}

\begin{proof}
The idea of the proof is to represent $LC_n=|LCS(X;Y)|$ as a 
function $f$ of a binomial random variable $N$, with $f$ satisfying
locally a reverse Lipschitz condition. 
Note that if a function $f:\mathbb{R}\rightarrow
\mathbb{R}$ 
satisfies a reverse Lipschitz condition, i.e., if $|f(x)-f(y)|\ge c|x-y|$,
for all $x,y\in\bbr$ or merely $x,y\in \calr (T)$, where $\calr  (T)$ is 
the range of a random variable $T$ with finite variance, then 
\begin{equation}\label{finitevar}
\var f(T)=
\frac12\, \E (f(T)-f(\tilde T))^2
\ge 
\frac{c^2}2\, \E (T-\tilde T)^2= c^2\var T, 
\end{equation}
where $\tilde T$ is an independent copy of $T$.
Moreover, if the reverse Lipschitz condition is only satisfied locally,
i.e., if, say, $|T-\tilde T|\ge r$, then \eqref{finitevar} is complemented by:
\begin{align}\label{compfinvar}
\var f(T)&\ge \frac{c^2}2\E(T-\tilde T)^2\mathbbm{1}_{|T-\tilde T|\ge r}\\
&= c^2\var T-\frac{c^2}2 \E (T-\tilde T)^2\mathbbm{1}_{|T-\tilde T| < r}\nonumber\\
&\ge c^2\var T-\frac{c^2}2r^2.
\end{align}
Therefore, if
$N$ is a binomial random variable with parameters $m=n/2d$ and $p$,
\begin{equation}
\label{VAFf}
\var f(N)
\geq \frac{c^2np(1-p)}{2d}.
\end{equation}
This last inequality  gives the desired order
for the variance of  $f(N)$, i.e., $\var f(N)=\Theta(n)$,
and it remains to find a way to represent $LC_n$
as $f(N)$, where $f$ is a function which typically increases
linearly.

This is done as follows: let $X(\l)$ denote a string of length
$n$ whose distribution is the distribution
of   $X$ conditioned on the number of long blocks to be $\l$:
$$\mathcal{L}(X(\l))=\mathcal{L}(X\mid N=\l),$$
where $\call$ stands for the law of the corresponding random variables.
The strings $X(\l)$ are all taken independent of $Y$ and of $N$.
We first simulate $X(m)$. For this, $X(m)$ is
a string of length $n$ with long blocks in every interval
$J_i$, for $i=1,2,\ldots,m$ and iid outside those intervals.
 Hence $X(m)$ is the string ``with the maximum number of long blocks
inserted". Then, we obtain $X(m-1)$ by choosing
in $X(m)$ one long block uniformly at random
 and turning it into iid and 
proceeding by induction,
once $X(\l)$ is defined we obtain $X(\l-1)$ by choosing 
uniformly at random one long block
in $X(\l)$ and turning it into iid. 
(We consider only the artificially
inserted long blocks, and not blocks in the iid part which might be long
by chance.) Next, let 
$$LC_n(\l):=|LCS(X(\l);Y)|.$$
It is  easy to
see that, with this construction, $X(\l)$ has the same
distribution as $X$ conditional on $N=\l$ and therefore that
$X(N)$ has the same distribution as $X$. So 
$LC_n$ has the same distribution as $LC_n(N)$ and 
$$\var LC_n=\var LC_n(N).$$
Now take $f: \l\mapsto LC_n(\l)$.
Note that by \eqref{bias}, 
$\l\mapsto LC_n(\l)$ behaves like a biased random walk path
which insures that the function $LC_n(\cdot)$ tends to increase linearly.
Clearly, $\l\mapsto LC_n(\l)$ is typically not going to increase
at every step but rather on a  $\log n$ scale.
This is enough to get an inequality like \eqref{VAFf},
by extending techniques developed in \cite{BM}, \cite{B}, or \cite{VARTheta}.
\end{proof}

\section{Replacing a long block
by iid has a biased effect}
\label{gail}
Next, we show that changing a randomly chosen long
block into iid has the desired biased effect in the current
multi-long block setting. The results
 of \cite{amsaluhoudrematzi2} 
where only one long block
 and both strings have length $2d$. 

By \cite{amsaluhoudrematzi2}, the probability, in the 
one-long-block setting
not to linearly  increase in the length of the long block,
is extremely unlikely as soon as $d$ is not too small. 
With this result for one long block, it should not come as 
a surprise that with many long blocks, most of them
if changed into iid lead to an increase of the LCS.
To make this argument rigorous there are two problems 
to overcome:
\begin{enumerate}
\item Our result for one long block assumes that both
sequences $X$ and $Y$ have length exactly equal to $2d$.
An optimal alignment $\vec a$ of sequences of length $n$,
will however ``map" the 
pieces 
\begin{equation}
\label{pieces}
X_1X_2X_3\cdots X_{2d},X_{2d+1}X_{2d+2}\cdots X_{4d},
X_{4d+1}X_{4d+2}\cdots X_{6d},\dots,X_{2d(m-1)+1}\cdots X_n
\end{equation}
to pieces of $Y$ of various lengths. Our solution to this first problem is to
show that with high probability for an optimal alignment
most of the pieces of $X$ given in \eqref{pieces} get aligned with pieces
 of $Y$ ``not too different  in length to $2d$'' (the ``not
too different'' will be specified later). This is done in Lemma 
\ref{lemmaK}.
\item If $\vec a$ is an optimal (random) alignment and, say, it aligns
the piece 
\begin{equation}
\label{Xi}
X_{2id+1}X_{2id+2}\cdots X_{2(i+1)d},
\end{equation} with
$$Y_{r_i+1}Y_{r_i+2}\cdots Y_{r_{i+1}},$$ then the distribution
of $Y_{r_i+1}Y_{r_i+2}\cdots Y_{r_{i+1}}$ is no longer iid but 
rather complicated and poorly
understood. Our result for the one long block case assumes
however the $Y$-string to be iid. Our solution to this second problem
is to specify for
each   piece  \eqref{Xi} which 
part of $Y$ it gets aligned to in a non-random manner. 
Hence, specify reals:
$r_0=0<r_1<r_2<r_3<\cdots<r_m=n$, and since these are
non-random, the various aligned parts become independent.
 We can then  use exponential inequalities to control the number of 
parts for which changing their long block into iid
produces the desired increase. Now, for one such specification,
the resulting alignment will hardly be optimal. 
But, the optimal alignment will typically be found
in a collection of such alignments with non-random 
constraints. We then show that typically
for each single alignment in the  whole collection
the property holds. (That is the property of 
the biased effect of the random change.)  
This is done in Lemma~\ref{lemmaQ}
\end{enumerate}
So, what remains to be done is to explore
the two  problems just described. 
To start, let us  introduce these non-random constrained alignments  via
a numerical example:
Take $d=3$ and $n=18$, and consider
the three intervals:
\begin{equation}
\label{intervals}
[1,2d]=[1,6],\;\;[2d+1,4d]=[7,12],\;\;[4d+1,n]=[13,18].
\end{equation}
 We are now going to specify the intervals to which these
intervals should get aligned.
For example we could align the first with $[1,7]$, the second
with $[8,11]$ and finally the third with $[12,18]$. Within
those constraints, we align in such a way to get a maximum number
of aligned letter pairs.
Hence, in our current example, we align a maximum number
of letter pairs of $X_1X_2\cdots X_6$ and $Y_1Y_2\cdots Y_7$, then
 a maximal number of letter pairs of
$X_7X_8\cdots X_{12}$ and $Y_8\cdots Y_{11}$ and
finally of $X_{13}\cdots X_{18}$ with $Y_{12}\cdots Y_{18}$.
The maximum number of aligned letter pairs  under these constraints
is therefore equal to
\begin{align*}
|LCS(X_1X_2\cdots X_6;Y_1\cdots Y_7)|&+
|LCS(X_7X_8\cdots X_{12};Y_8\cdots Y_{11})|\\
& +
|LCS(X_{13}X_{14}\cdots X_{18};Y_{12}\cdots Y_{18})|.
\end{align*}
Note that the three terms in the sum 
are independent.
Of course the alignment
defined in this way
is not necessarily an alignment corresponding to a LCS.
Indeed, let $k=2$, $n=12$
and let the sequences
$x=101010111111$ and $y=001010011110$. 
Then the alignment which
aligns $101010$ with $00$ and $111111$ with
$1010011110$ is given by:
\begin{equation}
\label{align62}
\begin{array}{c|c|c|c|c|c| c|c|c|c|c| c|c|c|c|c| c|c|c|c}
x& &1&0&1&0&1&0& & &1& &1& & &1&1&1&1&
\\\hline
y& & &0& &0& & & & &1&0&1&0&0&1&1&1&1&0
\end{array}\,.
\end{equation}
In fact, it corresponds to two alignments:
first the alignment aligning $X_1X_2\cdots X_6$ with $Y_1Y_2$:
\begin{equation}
\label{align62part1}
\begin{array}{c|c|c|c| c|c}
1&0&1&0&1&0
\\\hline
 &0& &0& &  
\end{array}
\end{equation}
and second the alignment aligning $X_{7}X_{8}\cdots X_{12}$ 
with $Y_3Y_4\cdots Y_{12}$:
\begin{equation}
\label{align62part2}
\begin{array}{c|c|c|c|c| c|c|c|c|c}
1& &1& & &1&1&1&1&
\\\hline
1&0&1&0&0&1&1&1&1&0
\end{array}\,.
\end{equation}
The alignment \eqref{align62} is obtained by
``concatenating" the alignments \eqref{align62part1} 
and \eqref{align62part2}. The ``score" of the alignment
\eqref{align62} is the sum of the scores of the alignment
\eqref{align62part1} and \eqref{align62part2}. Here, the
alignment \eqref{align62part1} aligned two letter
pairs while \eqref{align62part2} aligned six. Hence, 
when $X=x$ and $Y=y$, the score of the alignment
\eqref{align62} is
\begin{equation*}
|LCS(x_1x_2\cdots x_6;y_1y_2)|+
|LCS(x_7x_8\cdots x_{12};y_3y_4\cdots y_{12})|=
2+6=8.
\end{equation*}

Now, the alignment corresponding to the LCS is
\begin{equation}
\label{optimalalign}
\begin{array}{c|c|c|c|c| c|c|c|c|c| c|c|c|c|c| c|c}
x& & &1&0&1&0&1&0& &1&1&1&1& &1&1
\\\hline   
y& &0& &0&1&0&1&0&0&1&1&1&1&0& & 
\end{array}\,,
\end{equation}
and $LC_n=9$. In the alignment \eqref{optimalalign},
$x_1x_2\cdots x_6$ gets aligned with
$y_1y_2\cdots y_6$. 

More generally, let $n=2dm$, and let
$0=r_0<r_1<r_2<\cdots< r_{m-1} < r_m= n$ be integers.
Then, we study
the best alignment under the following $m$ 
 constraints:
\begin{itemize}
\item[1)] $X_1X_2\cdots X_{2d}$ gets aligned with $Y_1Y_2\cdots Y_{r_1}$
\item[2)] $X_{2d+1}X_{2d+2}\cdots X_{4d}$ gets aligned with 
$Y_{r_1+1}Y_{r_1+2}\cdots Y_{r_2}$
\item[3)] $X_{4d+1}X_{4d+2}\cdots X_{6d}$ gets aligned with 
$Y_{r_2+1}Y_{r_1+2}\cdots Y_{r_3}$
\item[]$\qquad\vdots$
\item[m)] $X_{2d(m-1)+1}X_{2d(m-1)+2}\cdots X_{n}$ gets aligned with
$Y_{r_{m-1}+1}Y_{r_{m-1}+2}\cdots Y_{n}$.\end{itemize}
The score of the best alignment under the above $m$ constraints,
denoted by
$LC_n(\vec{r})=LC_n(r_0,r_1,\ldots,r_{m-1},r_m)$, is equal to:
\begin{align*}
LC_n(\vec{r})&=LC_n(r_0,r_1,\ldots,r_{m-1},r_m)\\
&:=
\sum^{m-1}_{i=0}
|LCS(X_{2di+1}X_{2di+2}\cdots X_{2d(i+1)};Y_{r_i+1}
Y_{r_i+2}\cdots Y_{r_{i+1}})|.
\end{align*}
Let now 
\begin{align}
\label{tildeLr1rm}
\widetilde{LC}_n(\vec{r})&=\widetilde{LC}_n(r_0,r_1,\ldots,r_{m-1},r_m)
\nonumber\\
&= \sum^{m-1}_{i=0}
|LCS(\tilde X_{2di+1}\tilde X_{2di+2}\cdots \tilde X_{2d(i+1)};
Y_{r_i+1} Y_{r_i+2}\cdots Y_{r_{i+1}})|,
\end{align}
denote the score when the sequence $X$ is replaced by 
the sequence $\tilde{X}$. (Recall that the sequence $\tilde{X}$
is obtained from $X$ by replacing a randomly chosen long block
by iid.)

Let $\mathcal{R}^n$ be the set of all the partitions
of the integer interval $[0,n]$ into $m$ pieces:
$$\mathcal{R}^n:=\left\{(r_0,r_1,\ldots,r_m)\in[0,n]^{m+1}
: r_0=0\leq r_1\leq r_2\leq \cdots \leq r_m=n    \right\},$$
let $\gamma_k^e$ be a constant independent of $d$ such that
$$\gamma_k^e<\gamma^*_k,$$
and let $0<q^e<1$ be the unique real, which exists by concavity,
such that
$$\gamma_k(q^e)=\gamma_k^e.$$
Let $\epsilon>0$ and let
$$\mathcal{R}^n(\epsilon)\subset \mathcal{R}^n,$$
be the subset of those element of $\mathcal{R}^n$ which
have more than a proportion $1-{\epsilon p}/{2}$ of the values
$r_i-r_{i-1}$ in the interval
\begin{equation}
\label{interval}
\left[2d\frac{1-q^e}{1+q^e},2d\frac{1+q^e}{1-q^e}\right].
\end{equation}
More precisely,
$(r_0,r_1,\ldots,r_m)\in \mathcal{R}^n(\ep)$ 
if and only if 
$${\rm Card} \left\{i\in\{1,\ldots,m\}:
r_i-r_{i-1}\notin
\left[2d\frac{1-q^e}{1+q^e},2d\frac{1+q^e}{1-q^e}\right]
\right\}\leq \frac{mp\epsilon}{2}.$$
With these notations, we then proceed to prove
that with high probability every optimal alignment is in 
$\mathcal{R}^n(\epsilon)$. 

At this stage the reader, might wonder about the significance of the interval
\eqref{interval}. The answer is found when we consider
two independent iid strings where one has length $2d$
and the other has any length not in the interval \eqref{interval}.
Then, the expected length of the LCS of two such strings,
is at most $\gamma_k^e$ times the average of the
lengths of the two strings. 
Indeed, recall that when one sequence has length $2d$ and the other
has length $2ds$, then by definition \eqref{gammaknp},
\begin{equation}
\label{paloalto}
\gamma_k(d(s+1),q_s)=
\frac{\E(|LCS(X_1^*X_2^*\cdots X_{2d}^*;Y_1Y_2\cdots Y_{2ds})|)}{d(1+s)}
\end{equation}
with 
$$q_s=\frac{s-1}{s+1}.$$
When $s \not\in [(1-q^e)/(1+q^e),
(1+q^e)/(1-q^e)]$ (which corresponds to taking $s2d$ outside
the interval given in \eqref{interval}), then $q_s$
falls outside the interval $[-q^e,q^e]$. In other words, the expression
on the right side of \eqref{paloalto}
is equal to $\gamma_k(n,q_s)$ with $n$ being the average length of the 
strings and a $q_s\notin [-q^e,q^e]$. By subadditivity,
the limit always exceeds the current value:
$$\gamma_k(n,q)\leq \gamma_k(q)$$
for all $q\in[-1,1]$ and all $n$. Then by symmetry around zero
and by concavity, if $q_s\notin[-q^e,q^e]$,
then 
$\gamma_k(q_s)$ is strictly smaller than the value of $\gamma_k$
at the boundary of $[-q^e,q^e]$, hence
$$\gamma_k(q_s)<\gamma_k(q^e)=\gamma^e_k<\gamma^*_k.$$

We can now use  this for an alignment $\vec a$ between $X$ and $Y$.
 Since, $\gamma_k^e<\gamma^*_k$,
we infer that
 if too many of the pieces $X_{2di+1}X_{2di+2}\cdots
X_{2d(i+1)}$  are to be matched by $\vec a$
 with
a piece of $Y$ having length outside the interval given in \eqref{interval},
then the score of the alignment $\vec a$ would,
with high probability, be below optimal. Hence,
$\vec a$ would typically not correspond to an LCS.
This argument is made rigorous in the proof of
Lemma~\ref{lemmaQ}.

Denote by 
$K^n(\epsilon)$ the event
that every optimal alignment is in $\mathcal{R}^n(\epsilon)$,
i.e., 
$$K^n(\epsilon):=\left\{\forall \vec{r}\in\mathcal{R}^n:
LC_n(\vec{r})=LC_n, \vec{r}\in \mathcal{R}^n
(\epsilon)\right\}.$$
Let
\begin{equation}
\label{Deltavecr}\Delta(\vec{r})=\Delta(r_0,r_1,\ldots,r_m):=
\widetilde{LC}_n(r_0,r_1,\ldots,r_{m-1},r_{m})-LC_n(r_0,r_1,\ldots,r_{m-1},r_{m}).
\end{equation}
After showing that $K^n(\epsilon)$ has high probability,
it is also shown that with high probability, every
``alignment of $\mathcal{R}^n(\epsilon)$'' has a strong
conditional increase. For this we need to define the event $M^n(\ep)$.

Let $\vec{r}$ be an alignment of 
$\mathcal{R}^n$. Let $M^n_\epsilon(\vec{r})$ be the event
that among the  integers $i=1,2,\ldots,m$, there are less than
$mp\epsilon/2$ of them for which
the length $r_i-r_{i-1}$ is in the interval \eqref{interval},
and that there is a long block 
but for which the LCS of $X_{2d(i-1)+1}X_{2d(i-1)+2}\cdots X_{2di}$
with $Y_{r_{i-1}+1}Y_{r_{i-1}+2}\cdots Y_{r_i}$
does not increase by at least $\kappa d^\beta$ when
replacing the long block by iid. 
Hence, $M^n_\epsilon(\vec{r})$ is the event
that the set
$$\left\{i\in\{1,2,\ldots,m\}: Z_i=1,r_i-r_{i-1}\in \left[
2d\frac{1-q^e}{1+q^e}, 2d\frac{1+q^e}{1-q^e}\right],
\Delta(\vec{r})_i<\kappa d^\beta\right\}$$
contains less than $mp\epsilon/2$ elements.
Here, 
\begin{align*}
\Delta(\vec{r})_i&:=|LCS(X^*_{2d(i-1)+1}\cdots X^*_{2di};
Y_{r_{i-1}+1}\cdots Y_{r_{i+1}})|\\
&\qquad\qquad\qquad- |
LCS(X_{2d(i-1)+1}\cdots X_{2di}; Y_{r_{i-1}+1}\cdots Y_{r_{i+1}})|.
\end{align*}

Let $M^n(\epsilon)$ be the event that
$M^n_\epsilon(\vec{r})$ holds for every
$\vec{r}\in \mathcal{R}^n$:
$$M^n(\epsilon):=\bigcap_{\vec{r}\in\mathcal{R}^n}M_\epsilon^n(\vec{r}).$$

We will also need an event to insure that there are
enough long blocks. For this,
let $O^n$ be the event  that there are at least
$mp/2$ long blocks, i.e.,
$$O^n=\left\{\sum_{i=1}^mZ_i\geq \frac{mp}2\right\}.$$
The events $K^n(\ep)$, $M^n(\epsilon)$ and $O^n$, together imply the desired
expected conditional increase due to the random change of the 
long block into iid. This is the content of the next lemma.

\begin{lemma}
\label{combinatorics}
On $K^n(\ep)\cap M^n(\ep)\cap O^n$,
$$ \E\left.\left(|LCS(\tilde{X};Y)|-|LCS(X;Y)|
\right|X,Y\right)\geq
d^\beta (\kappa(1-2\epsilon)
-2\epsilon) .$$
\end{lemma}

\begin{proof} Let $\vec{a}$ be an optimal alignment of $X$ and $Y$.
Then, $\vec{a}$ can be viewed as a vector
$$\vec{a}=(a_0,a_1,\ldots,a_m)$$
with $a_0=0<a_1<a_2<\cdots<a_m=n$,
so 
\begin{equation}
\label{pieceX}
X_{2(d-1)i-1}\cdots X_{2di}
\end{equation} 
is aligned in an optimal way with
\begin{equation}
\label{pieceY}
Y_{a_{i-1}+1}\cdots Y_{a_i},
\end{equation} 
for $i=1,2,\dots, m$.
Hence, 
\begin{align*}
&|LCS(X_1\cdots X_n;Y_1\cdots Y_n)|=\\
&\quad =\sum_{i=1}^m |LCS(X_{2(d-1)i-1}\cdots X_{2di};Y_{a_{i-1}+1}\cdots Y_{a_i})|
\end{align*}
Now each of the pieces of string \eqref{pieceX} for $i=1,2,\ldots,m$
contains at most one long block.

If $K^n(\epsilon)$ holds, then $\vec{a}$ is an alignment
in the set $\mathcal{R}^n(\epsilon)$. Hence,
at most $\epsilon p m/2$ of the strings \eqref{pieceX} with $i=1,2,\ldots,m$
do not get aligned with a string \eqref{pieceY} of length belonging to 
\eqref{interval}. Since, by $O^n$ there are at least $pm/2$
long blocks, we get that the probability to chose a long block
inside a string \eqref{pieceX} for which the corresponding
\eqref{pieceY} has length outside \eqref{interval} is at most $\epsilon$.
Similarly, when $O^n$ and $M^n(\epsilon)$ both hold, then
the probability that the randomly chosen long block
is inside a string \eqref{pieceX} which when replaced by iid
does not lead to an increase of LCS of at least $d^\beta\kappa$
is no more than $\epsilon$. In other words,
with $K^n(\epsilon)$, $M^n(\epsilon)$ and $O^n$ all holding,
  the probability that the  chosen  long block increases
the alignment score of $\vec{a}$ by at least $d^\beta\kappa$,
is at least $1-2\epsilon$. 
Therefore,
\begin{equation}
\label{heini}
 \E\left.\left(|LCS(\tilde{X};Y)|-|LCS(X;Y)|
\right|X,Y\right)\geq \E(\Delta(\vec{a})\mid X,Y)\geq
d^\beta (\kappa(1-2\epsilon)
-2\epsilon ).
\end{equation}~\end{proof}

To be of any use, this increase needs to be strictly positive.
We will see that holding $\kappa>0$, fixed,
we can take $\epsilon>0$ as small as we want and the events $K^n(\ep)$
and $M^n(\epsilon)$ will
still have almost full probability, as long as $d$ is taken
large enough but fixed. 

The bias
\eqref{heini}, holds when $K^n(\epsilon)$, $O^n$
and $M^n(\epsilon)$ all hold, therefore
\begin{align}
\label{Pkappa}
&\P\left(\E\left.\left(|LCS(\tilde{X};Y)|-|LCS(X;Y)|
\right|X,Y\right)<
d^\beta (\kappa(1-2\epsilon)-2\epsilon\right)\nonumber\\
&\qquad\qquad\qquad\qquad \leq \P((K^n)^c(\epsilon))+\P((O^n)^c)+\P((M^n)^c(\epsilon)).
\end{align}
The purpose of
the next three lemmas is to show that the events
$\P((K^n)^c(\epsilon))$, $\P((O^n)^c)$ and $\P((M^n)^c(\epsilon))$ 
hold with small probability. 

\begin{lemma}
\label{lemmaK}
Let $k\in\mathbb{N}$, $k\ge 2$ and let $\gamma^e_k<\gamma^*_k$. 
Let $\epsilon>0$.  Let $0<p<1$. 
 Let $d$ be such that $(1+\ln 2d)/{2d}\le
(\gamma^*_k-\gamma^e_k)^2p^2\ep^2/{32}$. Then,
$$\P((K^n)^c(\epsilon))\leq \exp\left(
-\frac{n(\gamma^*_k-\gamma_k^e)^2p^2\ep^2}{32}
\right),$$
for all $n=2dm$, $m\in \mathbb{N}$.
\end{lemma}

\begin{proof}
Let $\vec{r}=(r_0,r_1,r_2,\ldots,r_m)$ be an alignment in 
$\mathcal{R}^n$.
Let $LC_n^*(\vec{r})$ denote the alignment score when
aligning $X^*$ with $Y$ according to $\vec{r}$:
\begin{equation*}
LC_n^*(\vec{r}):=
\sum^{m-1}_{i=0}
|LCS(X^*_{2di+1}X^*_{2di+2}\cdots X^*_{2d(i+1)};Y_{r_i+1}
Y_{r_i+2}\cdots Y_{r_{i+1}})|,
\end{equation*}
and let $LC^*_n$ denote the score of the LCS when aligning $X^*$
with $Y$:
$$LC^*_n:=|LCS(X_1^*X_2^*\cdots X_n^*; Y_1Y_2\cdots Y_n)|.$$
When the alignment $\vec{r}$ does not belong
to $\mathcal{R}^n(\epsilon)$, then for $n$ large enough, and as explained
at the end of the present proof,
\begin{equation}
\label{differenceexpectation}
\E (LC^*_n(\vec{r})-LC^*_n)\leq -\frac34(\gamma^*_k-\gamma_k^e)p\epsilon n.
\end{equation}
Next, recall that $LC_n(\vec{r})$ denotes the alignment score
when we align  $X$ with $Y$ according to $\vec{r}$,
while $LC_n:=|LCS(X;Y)|$.
The difference between $X^*$ and $X$ is at most
$m$ long block of length $d^\beta$. Hence the absolute difference between
$LC^*_n(\vec{r})$ and $LC_n(\vec{r})$ is at most
$md^\beta$ and so is the absolute difference
$|LC^*_n-LC_n|$. Therefore,
\begin{equation}
\label{mdbeta}
\left|LC^*_n(\vec{r})-LC^*_n-(LC_n(\vec{r})-LC_n)\right|
\leq 2d^\beta m.
\end{equation}
But if
\begin{equation}
\label{aaa}LC_n(\vec{r})-LC_n\geq 0,
\end{equation}
is to hold, then, by \eqref{mdbeta}, necessarily
\begin{equation}\label{eqinsert}
LC^*_n(\vec{r})-LC^*_n\geq -2d^\beta m.\end{equation}
Next, recall that $\beta<1$, and choose $d$ large enough
so that
\begin{equation}\label{aftermdbeta}
4d^\beta\le (\gamma^*_k-\gamma_k^e)p\epsilon d.\end{equation}
 Combining \eqref{eqinsert} with 
\eqref{differenceexpectation} and \eqref{aftermdbeta}
leads to:
\begin{equation}
\label{bbb}
LC_n^*(\vec{r})-LC_n^*-\E(LC^*_n(\vec{r})-LC^*_n)\geq 
\frac12(\gamma^*_k-\gamma_k^e)p\epsilon n,
\end{equation}
and therefore
\begin{equation}
\P(LC_n(\vec{r})-LC_n\geq 0)\leq
\P\left(LC_n^*(\vec{r})-LC_n^*-\E(LC^*_n(\vec{r})-LC^*_n)\geq 
\frac12(\gamma^*_k-\gamma_k^e)p\epsilon n\right).
\end{equation}
Since
 $LC^*_n(\vec{r})-LC^*_n$
depends on the iid random variables $X^*_1,X^*_2,\ldots,X^*_n$
and $Y_1,Y_2,\ldots, Y_n$, and  changes by at most $2$
when one of these variables changes,
Hoeffding's exponential martingale inequality (e.g., see 
\cite[Chap.~12]{GS}) ensures that:  
\begin{equation}
\label{boundexp}
\P\left(LC_n^*(\vec{r})-LC_n^*-\E(LC^*_n(\vec{r})-LC^*_n)\geq
\frac12(\gamma^*_k-\gamma_k^e)p\epsilon n\right)
\leq \exp\left(-\frac{n(\gamma^*_k-\gamma_k^e)^2p^2\epsilon^2}{16}\right).
\end{equation}
Now for $(K^n)^c(\epsilon)$ to hold, we need at 
least one  optimal alignment $\vec{r}\in\mathcal{R}^n$
which is not in $\mathcal{R}^n(\epsilon)$. But, if
$\vec{r}$ is optimal then it corresponds
to a LCS and thus $LC_n(\vec{r})-LC_n\geq 0$. Therefore,
$$(K^n)^c(\epsilon)
=\bigcup_{\vec{r}\notin\mathcal{R}^n(\epsilon)}
\left\{LC_n(\vec{r})-LC_n\geq 0\right\},
$$
so that
$$\P((K^n)^c(\epsilon))\leq
\sum_{\vec{r}\notin\mathcal{R}^n(\epsilon)}\P(LC_n(\vec{r})-LC_n\geq 0  )
.$$ 
This last sum contains at most
$\binom{n}{m}$ terms and so with the help of 
\eqref{boundexp}, 
\begin{align}
\label{Kncep}
\P((K^n)^c(\epsilon))&\leq \binom{n}{m}\exp
\left(-\frac{n(\gamma^*_k-\gamma_k^e)^2
\epsilon^2}{16}\right) \nonumber\\
&\le\left(\frac{ne}m\right)^m\exp
\left(-\frac{n(\gamma^*_k-\gamma^e_k)^2\epsilon^2}{16}\right)\nonumber\\
&\leq \exp
\left(-\frac{n((\gamma^*_k-\gamma_k^e)^2p^2
\epsilon^2)}{16}+\frac{n(1+\ln 2d)}{2d}\right),
\end{align}
since $n=2dm$.
Our choice of $d$, then leads to
$$
\P((K^n)^c(\epsilon))\leq \exp
\left(-\frac{n(\gamma^*_k-\gamma_k^e)^2p^2
\epsilon^2}{32}\right)
.
$$
 
Let us next detail
how the inequality \eqref{differenceexpectation} is obtained.
This inequality only holds for alignments $\vec{r}$ which are
not in $ \mathcal{R}^n(\epsilon)$.
Hence, assume now that $\vec{r}=(r_0,r_1,\ldots,r_m)\in\mathcal{R}^n$, but
that $\vec{r}\notin \mathcal{R}^n(\epsilon)$.
Then, there
 are at least $2mp\epsilon$ of the substrings
\begin{equation}
\label{Ystringi}
Y_{r_i+1}Y_{r_i+2}\cdots Y_{r_{i+1}},
\end{equation}
with length outside the interval \eqref{interval}.
For such a string, then as explained next,
 the expected value  with the corresponding
piece of $X^*$ is at most
$\gamma_k^e$ times half the number of symbols involved.
Thus, if the length of \eqref{Ystringi} is not in \eqref{interval},
then
\begin{equation}
\label{smaller0}
\E |LCS(X_{2di+1}^*X_{2di+2}^*\cdots X_{2d(i+1)}^*; 
Y_{r_i+1}Y_{r_i+2}\cdots Y_{r_{i+1}})|
\leq \frac{\gamma_k^e}2(2d+r_{i+1}-r_i )
\end{equation}
(To obtain this last inequality, use the fact that 
the expectation on the left-hand side of
\eqref{smaller0} is by definition of $\gamma_k(\cdot,\cdot)$ (see
\eqref{gammaknp})
equal to $\gamma_k(j,q^*)/2$ times the number of symbols $j$ involved,
where $q^*=(r_{i+1}-r_i-2d)/j$,
while $j=2d+r_{i+1}-r_{i}$. Moreover, the function
$t\rightarrow\gamma_k(t,q^*)$ is subadditive so that
\begin{equation}
\label{little}
\gamma_k(t,q^*)\leq \gamma_k(q^*),
\end{equation}
for all $t\in\mathbb{N}$.
When $r_{i+1}-r_{i}$ is outside the interval
\eqref{interval}, then $q^*$ is outside of $[-q^e,q^e]$.
But since the function $\gamma_k$  is symmetric around the origin
and concave,
\begin{equation}
\label{little2}
\gamma_k(q^*)\leq\gamma_k(q^e)=\gamma_k^e.
\end{equation}
Combining \eqref{little} and \eqref{little2}, leads to
$$\gamma_k(j,q^*)\leq \gamma_k^e,$$
which in turn leads to
\eqref{smaller0}.)

We can now apply a very similar argument for those $i$'s,
for which $r_{i+1}-r_{i}$ is in \eqref{interval}. For those
$i$'s, instead of \eqref{smaller0}, we find:
\begin{equation}
\label{smaller2}
\E\left|LCS(X_{2di+1}^*X_{2di+2}^*\cdots X_{2d(i+1)}^*; 
Y_{r_i+1}Y_{r_i+2}\cdots Y_{r_{i+1}})\right|
\leq 
\frac{\gamma^*_k}2(2d+r_{i+1}-r_i).\end{equation}
 Combining \eqref{smaller2} and \eqref{smaller0},
we have:
\begin{align}\label{0.5}
\E LC_n(\vec{r}) &= \sum^{m-1}_{i=0}
\E \left|LCS (X^*_{2di+1}X^*_{2di+2} \cdots X^*_{2d(i+1)};
Y_{r_{i}+1} Y_{r_{i}+2}\cdots Y_{r_{i+1}})\right|\nonumber\\
&\le \frac{\gamma^e_k}2
\sum^{m-1}_{\stackrel{i=0}{\vec r\notin\calr_n(\ep)}}
(2d+r_{i+1}-r_i)+\frac{\gamma^*_k}2
\sum^{m-1}_{\stackrel{i=0}{\vec r\notin\calr_n(\ep)}}
(2d+r_{i+1}-r_i)\nonumber\\
&=\frac{\gamma^*_k}2\sum^{m-1}_{i=0} (2d+r_{i+1}-r_i)+
\left(\frac{\gamma^e_k-\gamma^*_k}2\right)
\sum^{m-1}_{\stackrel{i=0}{\vec r\notin\calr_n(\ep)}}
(2d+r_{i+1}-r_i)\nonumber\\
&\le\frac{\gamma^*_k}2 (2dm+n)+
\left(\frac{\gamma^e_k-\gamma^*_k}2\right)
\sum^{m-1}_{\stackrel{i=0}{\vec r\notin\calr_n(\ep)}} 2d\nonumber\\
&\le\gamma^*_kn+
\left(\frac{\gamma^e_k-\gamma^*_k}2\right)2d\,2mp\ep\nonumber\\
&=\gamma^*_kn-(\gamma^*_k-\gamma^e_k)np\ep.
\end{align}
Next, as $n\rightarrow\infty$, 
$\E LC^*_n/n \to \gamma^*_k$. So, taking $n$ large enough,
\begin{equation}
\label{0.3}
\E LC^*_n \geq \gamma^*_k n-\frac n4(\gamma^*_k-\gamma_k^e)p\epsilon.
\end{equation} 
In fact, our conditions on $d$, imply that \eqref{0.3}
is, by  \eqref{iidseq}, satisfied for all $n=2dm$.
Combining \eqref{0.5}
and \eqref{0.3} gives
the desired inequality \eqref{differenceexpectation}.
\end{proof}
 
\begin{lemma}
\label{lemmaO} Let $m\in \N$, let $0<p<1$, then
$$\P(O^{n})\geq 1-\exp\left(-\frac{np^2}{4d}\right), $$ 
for all $n=2dm$, $d\in\N$.
\end{lemma}

\begin{proof}
The total number 
of long blocks $\sum_{i=1}^{m}Z_i$ is a binomial random variable
with parameters $m=n/2d$ and $p$. 
Thus,
\begin{equation*}
1-\P(O^n)=\P\left(\sum_{i=1}^{m}Z_i\leq \frac{mp}2\right)
=\P\left(\sum_{i=1}^mZ_i-\E\sum_{i=1}^mZ_i\leq -\frac{mp}2\right)
\leq \exp\left(-\frac{mp^2}2\right),
\end{equation*}
by Hoeffding's inequality.
\end{proof}

\begin{lemma}
\label{lemmaQ} Let $\ep >0$. Let $0<p<1$. 
Let $\frac12<\alpha <\beta <1$. Then, for $d$ large enough, 
$$\P((M^{n}(\epsilon))^c)\leq  \exp\left(-C_Md^{2\alpha-2\epsilon p}\right) 
,$$
for all $n=2dm$ and where $C_M>0$ is a constant independent of $d$, $n$,
$\epsilon$ and $p$.
\end{lemma}

\begin{proof} It is already shown in \cite{amsaluhoudrematzi2},
 in the one long block
situation, that changing the long block into iid tends to increase
the LCS-score. (See 
 Theorems~3.1 and 3.2 in \cite{amsaluhoudrematzi2}
and the events 
$H^d$ and $K^d$ there.) Now, these results are proved when the two 
strings have length exactly equal to $2d$. However, 
the same order of magnitude for the corresponding probability
holds true, if the sequence $Y$ has length in
the interval \eqref{interval} instead of exactly equal to $2d$.
(This is proved in Theorem~\ref{oneblocktheorem} of the Appendix.)
Let now $\vec{r}\in \mathcal{R}$. So, $\vec{r}=(r_0,r_1,\ldots,r_m)$ corresponds
to a specification of which parts of $Y=Y_1Y_2\cdots Y_n$
the different pieces
\begin{equation}
\label{lember}
X_{2d(i-1)+1}\cdots X_{2di},
\end{equation}
get aligned to. That is, 
for each $i=1,2,\ldots,m$, the string \eqref{lember}
gets aligned with 
$$Y_{r_{i-1}+1}Y_{r_{i-1}+1}\cdots Y_{r_i},$$
and the score thus obtained is denoted by 
$LC_n(\vec{r})$. Now,  by Theorem~\ref{oneblocktheorem}, for a single $i$ to be such
that $r_{i}-r_{i-1}$ is in the interval \eqref{interval},
but also such that replacing the long block by iid does not give
the increase of $\kappa d^\beta$, has a probability 
upper-bounded $\exp(-cd^{2\alpha-1})$, for some constant $c>0$. 
So, to have $\epsilon p m/2$
such intervals where the expected increase does not take
place would have a probability of less
than
$$ 
\exp\left(- cd^{2\alpha-1}\frac{\epsilon p m}2\right),$$
provided we specify which of the intervals fail to show that increase.
More precisely, for any integer subset $I\subset\{1,2,3,\ldots,m\}$,
counting $\epsilon p m/2$ elements in it,
we have
\begin{align}
\label{merdemerde}
&\P(\forall i \in I, 
|LCS(X^*_{2d(i-1)+1}\cdots X^*_{2di};Y_{r_{i-1}+1}\cdots Y_{r_i})\nonumber\\
&\quad-LCS(X_{2d(i-1)+1}\cdots X_{2di};Y_{r_{i-1}+1}\cdots Y_{r_i})|
< \kappa d^\beta\mid Z_1=1)\leq \exp\left(- cd^{2\alpha-1}\frac{\epsilon p m}2\right).
\end{align}
The above inequality is for a non-random prespecified
set $I$. There are at most $\binom{m}{\frac{\epsilon p m}2}$ such sets. 
So, the event that there
exists a set $I\subset\{1,2,3,\ldots,m\}$ counting $\epsilon pm/2$
elements so that for each $i\in I$, the increase is not there,
has probability upper-bounded by multiplying \eqref{merdemerde}
by $\binom{m}{\frac{\epsilon p m}2}$. 
But this is precisely the event $M^n_\epsilon(\vec{r})$. Hence,
\begin{equation}
\label{boundo}
\P(M^{n}_\epsilon(\vec{r})^c)\leq 
\binom{m}{\frac{\epsilon p m}2} \exp\left(-cd^{2\alpha-1}\frac{\epsilon p m}2\right)
\le \left(\frac{2e}{\epsilon p}\right)^{\frac{\ep pm}2}
\exp\left(-cd^{2\alpha-1}\frac{\ep pm}2\right).
\end{equation}
Next we need the bound for $M^{nc}_\epsilon$:
\begin{equation}
\label{vogel}
\P((M^{n}(\epsilon))^c)\leq \sum_{\vec{r}\in\mathcal{R}} 
\P(M^{n}_\epsilon(\vec{r})^c)
\end{equation}
using the bound \eqref{boundo} and since there are less
than $\binom{n}{m}$ elements in $\mathcal{R}^n$,
\eqref{vogel} becomes
\begin{equation}
\label{lac}
\P((M^{n}(\epsilon))^c) \leq 
\binom{n}{m}\left(\frac{2e}{\epsilon p} \right)^{\frac{\ep pm}2}
\exp\left(- cd^{2\alpha-1}\frac{\epsilon p m}2\right).
\end{equation}
$$\P((M^{n}(\epsilon))^c)\leq e^{\frac nu\left(1+\mbox{Ln }2d+\frac{\ep p}2
\left(1-\mbox{Ln }\frac{\ep p}2\right)-cd^{2\alpha-1}\frac{\ep p}2\right)},
$$
since $\binom{n}{m}\le\frac{n^m}{m!}\le\left(\frac{en}m\right)^m$, and
since $n=2dm$.
\end{proof}

{\bf Proof of Theorem~\ref{theoremfluct}.}
By Theorem~\ref{equivalence}, in order to prove that
$\var LC_n=\Theta(n)$ it is enough to show the high
probability of a bias as in \eqref{bias}.
Now \eqref{Pkappa} asserts that
the probability of the bias
\begin{equation}
\label{expectedincrease}
\P\left(\E\left( |LCS(\tilde{X}; Y)|-|LCS(X; Y)|
\big|X,Y\right)\le
d^\beta (\kappa(1-2\epsilon)-2\epsilon)
\right),
\end{equation}
is bounded above by
\begin{equation}
\label{KOQ}
\P((K^n)^c(\epsilon))+\P((O^n)^c)+\P((M^n)^c(\epsilon)).
\end{equation}
Lemmas~\ref{lemmaK}, \ref{lemmaO} 
and \ref{lemmaQ}, 
imply that the bound \eqref{KOQ}, is exponentially
small in $n$.
So, with probability close to one, the expected change \eqref{expectedincrease},
is larger than $d^\beta (\kappa(1-2\ep)-2\ep)$.
In order to apply Theorem~\ref{equivalence}, we would need a bias 
larger than $c_1d^\beta$, where $c_1>0$ can be any constant
not depending on $n$ and $d$.
 To achieve this, simply take $\epsilon>0$
small enough so that
\begin{equation}
\label{boundkappa}
\kappa(1-2\epsilon)-2\epsilon >0.
\end{equation} 
e.g., $0<\epsilon= \kappa/4(\kappa+1)$.
With this choice of $\epsilon$, 
the bound \eqref{boundkappa} is equal to $\kappa/2$
and, in turn,
the expected conditional increase
\eqref{expectedincrease} is at least equal to $\kappa d^\beta/2$,
with high probability.
Therefore, by Theorem~\ref{equivalence}, 
it follows
that $\var {LC_n}=\Theta(n)$, for
$d$ large enough but fixed.
This finishes the proof.\hfill\rule{0.5em}{0.5em}

\medskip
\centerline{\bf On the  choice of the constants}

There may be several optimal alignments of $X=X_1\cdots X_n$
and $Y=Y_1\cdots Y_n$. Chose any of them and denote it
by $\vec{a}$. (Hence, $\vec{a}$ is a random alignment.)
Let $0=R_0<R_1<\cdots<R_m=n$ be random variables so that
the optimal alignment $\vec{a}$ aligns the following piece
of $X$: 
\begin{equation}
\label{thatpiece}
X_{2d(i-1)+1}X_{2d(i-1)+2}\cdots X_{2di}
\end{equation} 
 to the following piece of $Y$:
\begin{equation}
\label{Ypiece}
Y_{R_{i-1}+1}Y_{R_{i-1}+2}\cdots Y_{R_i}
\end{equation}
for all $i=1,2,\ldots,m$.

In other words, the optimal alignment score, that is the LCS
is obtained by aligning the string \eqref{thatpiece} with the string
\eqref{Ypiece},
for all $i=1,2,\ldots,m$. Hence,  the length of the LCS is
\begin{align*}
|LCS(X;Y)|=&
|LCS(X_1\cdots X_n;Y_1\cdots Y_n)|\\
=&\sum_{i=1}^m 
|LCS(X_{2d(i-1)+1}X_{2d(i-1)+2}\cdots X_{2di};
Y_{R_{i-1}+1}Y_{R_{i-1}+2}\cdots Y_{R_i})|.
\end{align*}
In order to get 
the desired bias we only need to verify two things:

First that most pieces \eqref{thatpiece} get aligned to
a piece of $Y$ whose length is not too dissimilar. (More
precisely we want $R_{i+1}-R_{i}$ to be in the interval
\eqref{interval} for most $i=1,2,\ldots,m$.)
 The event $K^n(\epsilon)$ takes
care of this.

Second, that most of the pieces \eqref{thatpiece},
for $i=1,2,\ldots,m$, are such that if they contain a long block
and the long block gets changed to iid of the same length,
then this results in an increase of the LCS which is proportional
to the length of the long block. This is the event $M^n(\epsilon)$.

From $K^n(\epsilon)$ and $M^n(\epsilon)$ it follows, 
under choices of parameters, then
the high probability of the biased effect of replacing
a long block by iid.
Let us show next that these choices are not mutually exclusive. 

Recall that $q$ is a measure of how similar in length two sequences
which we align are. More precisely,
$$\gamma_k(n,q):=
\frac{\E(|LCS(V_1V_2\cdots V_{n-nq};W_1W_2\cdots W_{n+nq})|)}{n}$$
where $V_1,V_2,\ldots$ and $W_1,W_2,\ldots$ are two independent iid
sequences with $k$ equiprobable symbols,
and
$$\gamma_k(q):=\lim_{n\rightarrow\infty} \gamma_k(n,q).$$
So, given two iid strings one of length $j$ and the other of length
$\ell$, then by our very definition the expected length of the LCS is
$$\E(|LCS(X^*_1X^*_2\cdots X_j^*;Y_1Y_2\cdots Y_l)|)=
\frac{\ell+j}{2}\gamma_k\left(\frac{\ell+j}{2},q\right),$$
where
$$q=\frac{\ell-j}{\ell+j}.$$
When, the two strings under consideration have equal length then $q=0$.
Otherwise, their average length remaining constant, we have that
$q$ increases as their difference in length increases.
As already mentioned $q\mapsto \gamma_k(q)$ is concave and symmetric
around $q=0$ (see \cite{amsaluhoudrematzi2}) and we let $p_M>0$ be the largest real so
that $\gamma_k$ is constantly equal to its maximum on $[-p_M,p_M]$.
Let us next mention how we chose our parameters:
\begin{enumerate}
\item{}We first are going to determine the constant $q^e_A$ used in the proof
of Theorem~\ref{oneblocktheorem}.
For this we use the third of our conditions in Subsection \ref{conditions}.
By that condition and by continuity, we can 
 find $q^e_A>p_M$ so that
$$\left|\frac{\gamma'_k(q-)}{2}\right|<\left|\frac{\gamma_k(q)}{2}-\frac{1}{k}\right|$$
for all
$$q\in[-q^e_A,q^e_A].$$  This
condition which $q^e_A$ satisfies is then used in the Appendix
to prove that replacing the long block by iid has a biased effect
in the case of strings of length of order linear in $d$. See, for this
inequality, \eqref{regensdorff} and \eqref{driftwood} in the proof
of Theorem \ref{oneblocktheorem}.
\item{} Second, we chose $q^e$ to be any value which is strictly between
$p_M$ and $q^e_A$:
$$q^e\in(p_M,q^e_A),$$ 
and note that
$$\gamma_k(q^e)<\gamma_k(0)=\gamma_k(p_M).$$
Recall that $\gamma_k(q^e)$ is denoted by $\gamma_k^e$
and that $q^e$ is the value for which we show that
in an optimal alignment of $X=X_1\cdots X_n$ and 
$Y=Y_1\cdots Y_n$ most of the strings
\eqref{thatpiece} get aligned with a string 
\eqref{Ypiece} which has its length not too different from
the first string in the sense that the ``$q$ for the 
two strings'' is within $[-q^e,q^e]$. This is the content 
of the event $K^n(\epsilon)$.

\item{} The bias defined by $\kappa>0$ which does not depend
on $d$ in Theorem \ref{oneblocktheorem} is determined
as soon as $q^e$ and $q^e_A$ are given. (See the proof
of Theorem \ref{oneblocktheorem}.) 
Given $\kappa>0$ which defines the 
positive bias in the one long block situation of Theorem
\ref{oneblocktheorem} we can now determine $\epsilon>0$
which would give a bias in the multi-long block situation.
As a matter of fact, by Lemma~\ref{combinatorics},
as soon as the events $K^n(\epsilon)$, 
$M^n(\epsilon)$ and $O(\epsilon)$ all hold, the expected increase in LCS
obtained by replacing a randomly chosen long block by iid
is at least
\begin{equation}
\label{swimming}
d^\beta(\kappa(1-2\epsilon)-2\epsilon).
\end{equation}
The lower bound above is only valuable if it is positive.
This can be obtained by  simply choosing $\epsilon>0$ to satisfy:
\begin{equation}
\label{COND2}\frac{\kappa}{2+2\kappa}>\epsilon.
\end{equation}
The inequality \eqref{COND2} holding, ensures that a 
strictly positive 
bias holds for the expected increase obtained by replacing a randomly chosen
long block by iid. Now, such a positive bias needs to hold with
high probability in order to get the main result of this paper.
(See Theorem \ref{equivalence}.) But, Lemma \ref{combinatorics}
shows that the positive bias of size at least \eqref{swimming}
is given as soon as all three events: $K^n(\epsilon)$, $M^n(\epsilon)$
and $O^n(\epsilon)$ all hold. 
So, the only thing
we need to check, is that these events hold with high probability
also for the parameters $\epsilon>0$ and $q^e>p_M$ which 
we have   determined so far. Indeed, $\epsilon$ and $q^e$
have already been determined and can no longer
be chosen freely. 
But, for any value of $\epsilon>0$ and $q^e>p_M$,
the event $K^n(\epsilon)$, for $d$ large enough,
will hold with high probability, and so will $O^n$ hold. 

For any given $\epsilon>0$ fixed, the event
$M^n(\epsilon)$, 
 holds with high probability
for $d$ large enough, as soon as the bias
effect holds in the one-long block case with high probability.
But this is guaranteed by Theorem \ref{oneblocktheorem}
and we have already chosen the parameters $q^e$ 
so that Theorem \ref{oneblocktheorem} holds.
\end{enumerate}

\section{Appendix}
\subsection{Outline of proof for biased effect in one-long-block-only case}
The current paper relies on results from 
\cite{amsaluhoudrematzi2} showing that in the one-long block
situation the effect of replacing the long block with iid
is typically linear in the length of the long block. However,
the result is formulated
with the two strings having both length exactly equal to $2d$.
For our current purpose, we need 
the first sequence
to have length exactly $2d$, but the second sequence
to have length close to $2d$ instead of exactly $2d$.
The proof, in this slightly more general situation,
is very similar to the one provided in \cite{amsaluhoudrematzi2},
so we  provide its quick outline.
Let us first formulate our theorem which is a slight generalization
of a corresponding result in \cite{amsaluhoudrematzi2}:

\begin{theorem}
\label{oneblocktheorem}
Let $1/2<\alpha<\beta<1$.
There exist $\kappa>0$ independent of $d$, 
so that provided $\gamma^e_k$ is close enough
to $\gamma^*_k$ (independently of $d$), we have
for all $d$ large enough:
\begin{align*}
\P&(|LCS(X_1^*\cdots X_{2d}^*;Y_1\cdots Y_i)|\\
&-|LCS(X_1\cdots X_{2d};Y_1\cdots Y_i)|
\geq \kappa d^\beta\mid Z_1=1)\geq 1-\exp(-cd^{2\alpha-1})
\end{align*}
for all $i$ contained in the interval \eqref{interval}, and
some absolute constant $c>0$.
\end{theorem}

\begin{proof} This theorem is proved  
for a slightly more restricted situation as
Theorem 3.1 and Theorem 3.2 in \cite{amsaluhoudrematzi2}.
 There the second sequence
must have length exactly $2d$, while here we allow it
to have any length in the interval given in \eqref{interval}. 
However, in \cite{amsaluhoudrematzi2},
one also considers the effect of adding a piece of iid
to two strings which have length approximately $d$ rather than
exactly $d$. Since this is the crucial point, 
let us nonetheless show 
a shortened version of its proof for the current generalized case.

When for two sequences $X$ and $Y$ close to  $2d$, we replace an iid part
by a long constant block in one of them, this causes an expected loss
of the LCS. The long 
block has length $d^\beta$, were $1/2<\beta<1$ does not depend on
$d$.  
The variance cannot make up for the expected loss  since, 
by Hoeffding's inequality, 
the standard deviation is at most of order $\sqrt{d}$.  But, 
$d^\beta$ ($\beta>1/2$), the length of the long constant block, has an order 
of magnitude greater that $\sqrt{d}$. So, we need to estimate the expected
loss due to the long block. Here is how it is done:
the string $X$ of length $2d$
is iid except that in the middle there is a long block
of only one letter (we work on the probability conditional
on $Z_1=1$). This means
that there are three strings $X^a$, $B$ and $X^c$,
where $X^a$ and $X^c$ are iid strings and $B$ is
a block of length $d^\beta$. The concatenation
of the three string gives: 
$$X^a B X^c,$$
where $X^a$ and $X^c$ have same length equal to: $d-d^\beta=d+o(d)$.
The next ingredient is $Y$ which is an iid string of length $i$,
where $i$ is in the interval \eqref{interval}.
Hence $Y=Y_1Y_2\cdots Y_{i}$,
and let $0<i_1<i_2<i$, be integers, so that
$$Y^a=Y_1\cdots Y_{i_1},\quad 
Y^b=Y_{i_1+1}Y_{i_1+2}\cdots Y_{i_2},\quad 
Y^c=Y_{i_2+1}Y_{i_2+2}
\cdots Y_i.$$
 Again, in all these strings and substrings,
except for the long block of length $d^\beta$, we have $k$ equiprobably letters.
Let $\pi$ be an alignment, and let $Y^a$, $Y^b$,
and $Y^c$ denote the pieces respectively aligned with
 $X^a$, $X^b$, and $X^c$.

Next, we  modify the alignment $\pi$ to obtain a new alignment
$\bar{\pi}$ where the long block has been replaced
by the iid part of same length $X^b$. 
For this new alignment $\bar{\pi}$
do the following:
 align $X^a$ with $Y^aY^b$ instead of only with 
$Y^a$. The block $B$ gets replaced by $X^b$
and then $X^b$ is ``added to the alignment of $X^c$ with $Y^c$".
So, we  request that $\bar{\pi}$ aligns $X^a$ and $Y^aY^b$
in an optimal way and also that $X^bX^c$ gets aligned optimally
to $Y^c$. Thus the part of the alignment
score of $\bar{\pi}$, coming from aligning $X^a$ with $Y^aY^b$,
is equal to $|LCS(X^a;Y^aY^b)|$, while the second part of that alignment
yields a score of
$$|LCS(X^bX^c;Y^c)|.$$

 Schematically the two alignments are represented as:
$$\pi:\quad \begin{array}{c|c|c}
X^a &B &X^c\\\hline
Y^a &Y^b & Y^c
\end{array}
$$ 
and
$$\bar{\pi}:\quad \begin{array}{c|c|c}
X^a & &X^bX^c\\\hline
Y^aY^b & & Y^c
\end{array},
$$ 
with alignment scores:
\begin{align*}
&{\tt score \; of\;}\pi=|LCS(X^a;Y^a)|+|LCS(B;Y^b)|+|LCS(X^c;Y^c)|\\
&{\tt score \; of\;}\bar{\pi}=|LCS(X^a;Y^aY^b)|+|LCS(X^bX^c;Y^c)|.
\end{align*}
The difference between the two alignment scores has two sources:
first the loss of the  aligned letter pairs of the block $B$ which where aligned
with letters (and not with gaps) under $\pi$.

Second, the gain  due to 
 ``adding $Y^c$ to the alignment of $X^a$ with $Y^a$'' and 
``adding $X^b$ to the alignment of $X^c$
with $Y^c$''. 
Since there are $k$-equiprobable letters and
 $B$ consisting only of one letter, $|LCS(B,Y^b)|$
is the number of times that letter appears in the string $Y^b$.
Hence, 
\begin{equation}
\label{lulo}
\E|LCS(B;Y^b)|=|Y^b|/k
\end{equation}
Now the expected change in score
is
\begin{align*}
\E({\tt score\; of\;}\bar{\pi}-{\tt score\;of\;}\pi)&=
\E(|LCS(X^a;Y^aY^b)|-|LCS(X^a;Y^a)|)\\
&\quad +
\E(|LCS(X^bX^c;Y^c)|-|LCS(X^c;Y^c)|)\\
&\quad\qquad\qquad\qquad\qquad\qquad -
\E(|LCS(B;Y^b)|).
\end{align*}
The two first terms in the sum on the right side of  the last
 equation above can be lower-bounded using Lemma \ref{addinstring} proved below.
Together with \eqref{lulo}, this yields
\begin{align}\label{eins}
&\E({\tt score\; of\;}\bar{\pi}-{\tt score\;of\;}\pi)\\ \label{zwei}
&\geq
\frac{|Y^b|}{2}\left( \gamma_k(p_I^a)-|\gamma_k'(p_{II}^a-)|-\frac{2}{k}
\right)+
\frac{|X^b|}{2}\left( \gamma_k(p_I^b)-|\gamma_k'(p_{II}^b-)|\right)-
O(\sqrt{d}\ln d), 
\end{align}
where
\begin{equation}
\label{pIapIIa}p_I^a=\frac{|Y^a|-|X^a|}{|Y^a|+|X^a|},\quad 
p_{II}^a=\frac{|Y^a|+|Y^b|-|X^a|}{|Y^a|+|Y^b|+|X^a|},
\end{equation}
and
\begin{equation}
\label{pIbpIIb}
p_I^c=\frac{|X^c|-|Y^c|}{|X^c|+|Y^c|},\quad 
p_{II}^c=\frac{|X^c|+|X^b|-|Y^c|}{|X^c|+|X^b|+|Y^a|}.
\end{equation}
As in \cite{amsaluhoudrematzi2}, 
we now consider two cases:

{\bf Case I:} $\gamma_k(0)/2>1/k$. In that case Condition 3
of Subsection~\ref{conditions},
is equivalent to 
\begin{equation}
\label{regensdorff}
 \frac{\gamma_k(q)}{2}-\frac{1}{k}-\frac{|\gamma_k'(q+)|}{2}>0,
\end{equation}
for all $q\in[-p_M,p_M]$.

Now uniform continuity and the right continuity
of the derivative of a convex function, imply that
\eqref{regensdorff} holds on an interval which is even
slightly bigger than $[-p_M,p_M]$ and also if the two entries of $q$ 
are not exactly equal but just very close. Formally, there exists
$\delta>0$ and $q_A^e$ with $p_M<q_A^e$ 
 so that:
 $\forall q_1,q_2\in[-q_A^e,q_A^e]$ with $|q_1-q_2|\leq \delta$
\begin{equation}
\label{driftwood}
 \frac{\gamma_k(q_1)}{2}-\frac{1}{k}-\frac{|\gamma_k'(q_2-)|}{2}>0.
\end{equation}
Now, for large enough $d$, 
$p_I^a$ and $p_{II}^a$ are close to each other, while
$p_{I}^b$ is close to  $p_{II}^b$. Indeed,
\begin{equation}
\label{starbucks}
|p_I^a-p_{II}^a|\leq \frac{2|Y^c|}{|Y^a|+|X|^a}=O\left(\frac{d^\beta}{d}\right),\quad
|p_{I}^b-p_{II}^b|\leq \frac{|B|}{|X^c|+|Y^c|}=O\left(\frac{d^\beta}{d}\right).
\end{equation}
These last two inequalities follow from that
for any $s,r,t>0$,
$$\left|\frac{s+r-t}{s+r+t}-\frac{s-t}{s+t}\right|=
\left|\frac{r}{s+t}\left(1+\frac{s+r-t}{s+r+t} \right)
\right|\leq \frac{2r}{s+t}.$$
Then taking
$s=|Y^a|$, $t=|X^a|$ and $r=|Y^b|$ yields the first
inequality in \eqref{starbucks}. The
second is obtained similarly.

We now define the event $B^d$ to be the event that 
for any optimal alignment of $X^aBX^b$
with $Y_1\cdots Y_i$, the corresponding parameters
$p_I^a$, $p_{II}^a$, $p_I^b$, and $p_{II}^b$ satisfy:
\begin{equation}
\label{antonio}
p_I^a,p_{II}^a,p_{I}^b,p_{II}^b\in[-q_A^e,q_A^e].
\end{equation}
In other words, the event
$B^d$
holds if
for all $i_1,i_2\in[0,i]$  with $i_1<i_2$ such that
\begin{align*}
|LCS(X^aBX^c;Y_1\cdots Y_i)|&=
|LCS(X^a;Y_1\cdots Y_{i_1})|+|LCS(B;Y_{i_1+1}\cdots Y_{i_2})|\\
&\qquad\qquad\qquad +
|LCS(X^c;Y_{i_2+1}\cdots Y_i)|,
\end{align*}
we have that \eqref{antonio} holds, for the values of 
$p_I^a$, $p_{II}^a$, $p_I^b$ and $p_{II}^b$ as defined by
\eqref{pIapIIa} and \eqref{pIbpIIb}.

We leave it to the reader to prove that the event $B^d$ holds 
with a probability close to one up to an exponential small quantity
in $d$. The proof is very similar to the proof that 
the event $K^n(\epsilon)$ holds with high probability.
In other words, \eqref{antonio} holds, for when ``$\pi$ is an optimal
alignment, rather than an alignment defined by non-random constaints
as we have done so far''.
(Here the non-random constraints are that $X^a$ should be aligned
with $Y_1\cdots Y_{i_1-1}$ and $X^c$ should be aligned with $Y_{i_2}\cdots Y_i$.)
Hence, by \eqref{driftwood} and \eqref{starbucks} and provided that 
\eqref{antonio} holds,
we find that for $d$ large enough,
\begin{equation}
\label{driftwoodII}
 \frac{\gamma_k(p_I^a)}{2}-\frac{1}{k}-\frac{|\gamma_k'(p_{II}^a+)|}{2}>0,
\end{equation}
and 
\begin{equation}
\label{driftwoodIII}
 \frac{\gamma_k(p_I^b)}{2}-\frac{|\gamma_k'(p_{II}^b+)|}{2}>
\frac{1}{k}.
\end{equation}
We can now apply \eqref{driftwoodII} and \eqref{driftwoodIII}
 to \eqref{zwei} in order to find:
\begin{equation}\label{matzinger}
\E({\tt score\;of\;}\bar{\pi}-{\tt score\;of\;}\pi)\geq
\frac{|X^b|}{2k}-
O(\sqrt{d}\ln d).
\end{equation}
This last inequality is basically ``the bias needed
for when replacing the long block by iid''. Indeed,
$X^b$ having the same length as the long block, we have
$|X^b|=d^\beta$. So, indeed \eqref{matzinger}
shows an increase in score by a linear quantity in the long
block, for when we replace the long bock by iid.
The only problem remaining is that so far we have the inequality 
only for an alignment $\pi$ for which $i_1$ and $i_2$ are 
non-random. But, we need it for an optimal alignment,
for which ``$i_1$ and $i_2$ are random''.
  To overcome this difficulty, proceed as usual,
showing that for all $i_1,i_2$ (non-random) in a suitable interval,
 \eqref{zwei} holds. Then,  the optimal (random)
alignment  will typically fall into one of these 
``non-random values''.
This allows us to have  \eqref{matzinger}  also verified for 
when $\pi$ is the optimal alignment.  Let us make 
that argument more precise:
Let $I_1$ and $I_2$ denote the ``$i_1$ and $i_2$ of
an optimal alignment''. Hence, assume
that with probability $1$, we have
\begin{align*}
|LCS(X^aBX^c;Y_1\cdots Y_i)|=
|LCS(X^a;Y_1\cdots Y_{I_1-1})|&+|LCS(B;Y_{I_1}\cdots Y_{I_2-1})|\\
&+
|LCS(X^c;Y_{I_2}\cdots Y_i)|.
\end{align*}
Note that, so far, the alignment score of $\pi$ and $\bar{\pi}$
depends on the non-random entries $i_1<i_2$. Thus define
$f(i_1,i_2)$ to be the expected increase in alignment score:
$$f(i_1,i_2):=\E(({\tt score\; of\;}\bar{\pi}-{\tt score\;of\;}\pi\;)_{(i_1,i_2)}),$$
and $g(i_i,i_2)$ to be the difference:
$$g(i_1,i_2):=({\tt score\; of\;}\bar{\pi}-{\tt score\;of\;}\pi\;)_{(i_1,i_2)}-
f(i_1,i_2).$$
By definition, the random variables $I_1,I_2$ are such,
that when they replace $i_1,i_2$, then the score of $\pi$
becomes the LCS of the sequence with a long block:
$${\tt score\;of\;}\pi_{(I_1,I_2)}=|LCS(X^aBX^c;Y_aY^bY^c)|.$$
On the other hand, when replacing $i_1,i_2$ by
$I_1,I_2$, the alignment  $\bar{\pi}$
is not necessarily an optimal alignment, but its score
is then at most the optimal alignment-score.
So ${\tt score\;of\;} \bar{\pi}_{(I_1,I_2)}$ is
 a lower bound on the optimal alignment score
of $X^aX^bX^c$ with $Y^aY^bY^c$. Therefore,
the change in optimal
alignment score due to replacing the
long block by iid is bounded below in the following manner:
\begin{align}
&|LCS(X^aX^bX^c;Y^aY^bY^c)|-|LCS(X^aB X^c;Y^aY^bY^c)|\nonumber\\
&\label{fi1i2}\geq {\tt score\;of\;}\bar{\pi}\;_{(I_1,I_2)}-{\tt score\;of\;}\pi\;_{(I_1,I_2)}=f(I_i,I_2)+g(I_1,I_2).
\end{align}
On the event $B^d$ holds, then we can use inequality
\eqref{matzinger}, and the law of total probability for expectation,
gives
\begin{equation}
\label{samutsporting}
\E(f(I_1,I_2))\geq \P(B^d)\left(\frac{|X^b|}{2k}-
O(\sqrt{d}\ln d)\right)-\P((B^d)^c)2\left(|X^b|+|Y^b|  \right).
\end{equation}
In obtaining \eqref{samutsporting} we also used
that between the two alignments $\pi$ and $\bar{\pi}$,
only $X^b$ and $Y^b$ change in terms of what they get aligned to,
so the maximum possible change between the score of $\pi$
and the score of $\bar{\pi}$ is $2(|X^b|+|Y^b|)$.
The quantity $\P((B^d)^c)$ is exponentially small in 
$d$. Hence for $d$ large enough $\P(B^d)\geq 1/2$.
Also, $|X^b|=d^\beta$ and $|Y^b|$ is at most of linear order
in $d^\beta$. Taking the expectation in \eqref{fi1i2}
and using \eqref{samutsporting}, give
\begin{align} 
\label{ali}
\E(|LCS(X^aX^bX^c;Y^aY^bY^c)|&-|LCS(X^aB X^c;Y^aY^bY^c)|)\nonumber\\
&\qquad\qquad\qquad\ge \frac12\left(\frac{|X^b|}{2k}-
O(\sqrt{d}\ln d)\right)+\E(g(I_1,I_2)).
\end{align}
We can assume that for some constant $K>0$, not depending on
$d$, we have $i\leq Kd$. Now, by Hoeffding's inequality,
 for all $i_1,i_2\leq i$ and all $\Delta>0$,
\begin{equation}
\label{lake}
\P\left(|g(i_1,i_2)|\geq \Delta\ln d\sqrt{d}\right)\leq 2\exp(-c
\Delta^2 (\ln d)^2),
\end{equation}
where $c>0$ is a constant which depends neither on $d$ nor on $\Delta$.
Since $i\leq Kd$,
by \eqref{lake}, 
\begin{align*}
\P\left(\max_{i_1,i_2\leq i}|g(i_1,i_2)|\geq \Delta\ln d\sqrt{d}\right)&\leq 
2K^2d^2\exp(-c \Delta^2 (\ln d)^2)\\
&=2K^2\exp(-c\Delta^2\ln^2 d+2\ln d).
\end{align*}
Assuming $\Delta>1$ and assuming $d$ sufficiently large
so that $c\Delta^2 \ln d\geq 2+\Delta^2 $,
$$\P\left(\max_{i_1,i_2\leq i}|g(i_1,i_2)|\geq \Delta\ln d\sqrt{d}\right)\leq 
2K^2\exp(-\Delta^2),$$
thus
$$\E\left(\max_{i_1,i_2\leq i}|g(i_1,i_2))|\right)\leq 
O(\sqrt d\ln d ),$$
and also
$$\E(|g(I_1,I_2)|)\leq O(\ln d\sqrt{d}).$$
The last inequality in \eqref{ali}
gives
\begin{align}
\E(|LCS(X^aX^bX^c;Y^aY^bY^c)|&-|LCS(X^aB X^c;Y^aY^bY^c)|)\nonumber\\
&\quad \geq\frac12\left(\frac{|X^b|}{2k}-
O(\sqrt{d}\ln d)\right),
\end{align}
which is the bias we seeked since $|X^b|=d^\beta$.
\end{proof}

\begin{lemma}
\label{addinstring} Let $s,t,r>0$.
Let three iid strings, with $k$ equiprobable letters, be given by:
$$X^I:=X^*_1X^*_2\cdots X^*_t,$$
$$Y^I=Y_1\cdots Y_s,$$
$$Y^{II}:=Y_{s+1}\cdots Y_{s+r}.$$
Let $r$ be of smaller order than $s+t$: $r=o(s+t)$.
Then, when ``adding the string $Y^{II}$ to the alignment of $X^I$
with $Y^{II}$'', the expected increase in LCS is such that
\begin{align}\label{eqnlem1}
\E(|LCS(X^I;Y^IY^{II})|-|LCS(X^I;Y^I)|)&\geq 
r\left(\frac{\gamma_k(p_I)}{2}-\frac{|\gamma'(p_{II}-)|}{2} \right)\nonumber\\
&\qquad +
O(\sqrt{s+t}\ln(s+t)),\end{align}
where the proportions of the length of the strings
are defined as 
$$p_I:= \frac{s-t}{s+t},\quad p_{II}:=\frac{s+r-t}{s+r+t} 
.$$
\end{lemma}

\begin{proof}
By the very definition of the function $\gamma_k(\cdot ,\cdot )$ given in
\eqref{gammaknp}, 
\begin{equation}
\label{odin}
\E(|LCS(X^I;Y^IY^{II})|-|LCS(X^I;Y^I)|)=
n_{II} \gamma_k(n_{II},p_{II})-n_{I} \gamma_k(n_I,p_I).
\end{equation}
Now, from Alexander \cite{Alexander,alexander2},
$\gamma_k(n,p)$
converges at a rate $C\ln n /\sqrt{n}$ to $\gamma_k(n,p)$
where $C>0$ is an absolute constant.
Hence, the right side of \eqref{odin} becomes:
\begin{equation}
\label{dva}
n_{II}\gamma_k(n_{II},p_{II})-n_{I}\gamma_k(n_I,p_I)
=n_{II}\gamma_k(p_{II})-n_{I}\gamma_k(p_I)+O(\sqrt{s+t}\ln(s+t)).
\end{equation}
Note that 
\begin{equation}
\label{tre}
n_{II}\gamma_k(p_{II})-n_{I}\gamma_k(p_I)=
  n_{II}\frac{\Delta\gamma}{\Delta p}\Delta p
+\Delta n \gamma_k(p_I),
\end{equation}
where
$$p_I=\frac{s-t}{s+t},\quad p_{II}=\frac{s+r-t}{s+r+t},\quad \Delta p=p_{II}-p_I,$$
and
$$n_I=\frac{t+s}{2},\quad n_{II}=\frac{t+s+r}{2},\quad \Delta n=n_{II}-n_I=\frac{r}{2},$$
and also $\Delta\gamma=\gamma_k(p_{II})-\gamma_k(p_I)$.

{\bf First case: $s\geq t$.} Then $0<(s-t)/(s+r+t)\leq p_I$, and therefore
$$0<\Delta p<p_{II}-\frac{s-t}{s+r+t}=\frac{r}{s+r+t},$$ 
and 
$$0<\Delta p\leq \frac{r}{s+r+t}.$$
This last inequality implies that $n_{II}\Delta p\leq r/2$,
and so the right-hand side of \eqref{tre} becomes:
\begin{equation}
\label{chitirie}
 n_{II}\frac{\Delta\gamma}{\Delta p}\Delta p
+\Delta n \gamma_k(p_I)\geq
 \frac{r}{2}\left(\gamma_k(p_I)-|\gamma'(p_{II}-)|\right).
\end{equation}
In obtaining \eqref{chitirie}, 
the fact that $\gamma_k$ is concave and symmetric about the origin while
$p_{II}>p_{I}>0$, imples that:
$$\left|\frac{\Delta \gamma_k}{\Delta p}\right|\leq 
\gamma_k'(p_{II}-).$$              
Then, \eqref{odin}, \eqref{dva}, \eqref{tre}
and \eqref{chitirie} jointly imply
the desired result:
$$\E(|LCS(X^I;Y^IY^{II})|-|LCS(X^I;Y^I)|)
p\geq 
\frac{r}{2}\left(\gamma_k(p_I)-|\gamma'(p_{II}-)|\right)
+O\left(\sqrt{s+t}\ln(s+t)\right).$$

{\bf Second case:} $s<t$ and $s+t<r$. In this case $p_I\leq p_{II}<0$.
Since $\gamma_k$ is concave and symmetric about the origin, 
$\Delta\gamma>0$. This then leads to
$$n_{II}\gamma_k(p_{II})-n_{I}\gamma_k(p_I)=
  n_{II}\Delta\gamma
+\Delta n \gamma_k(p_I)\geq \Delta n\gamma_k(p_I).$$
In turn, together with \eqref{odin}, \eqref{dva} implies:
$$\E(|LCS(X^I;Y^IY^{II})|-|LCS(X^I;Y^I)|)
\geq 
\frac{r}{2}\gamma_k(p_I)
+O\left(\sqrt{s+t}\ln(s+t)\right),$$
which implies \eqref{eqnlem1}.

{\bf Third case:} $s<t$ and $s+r>t$. In that case, $p_I<0$ while
$p_{II}>0$. Again since $\gamma_k$ is concave and symmetric around the origin,
it is non-decreasing from $p_I$ to $0$, and non-increasing from
$0$ to $p_{II}$. Hence,
\begin{equation}
\label{superbagel}
n_{II}\gamma_k(p_{II})-n_{I}\gamma_k(p_I)=
  n_{II}\Delta \gamma
+\Delta n \gamma_k(p_I)\geq 
n_{II}\frac{\gamma_k(p_{II})-\gamma_k(0)}{p_{II}-0}p_{II}+\Delta n\gamma_k(p_I).
\end{equation}
Now, 
\begin{equation}
\label{littlelittle}
0<n_{II}p_{II}=\frac{s+r-t}{2}\leq\frac{r}{2},
\end{equation}
since $s-t<0$. Applying 
\eqref{littlelittle} to \eqref{superbagel}
finally yields
\begin{equation}
\label{lalala}
n_{II}\gamma_k(p_{II})-n_{I}\gamma_k(p_I)\geq
\frac{r}{2}(\gamma_k(p_I)-|\gamma_k'(p_{II}-)|  ).
\end{equation}
Again, by concavity and symmetry,
$|\gamma_k(p_{II})-\gamma_k(0)|/p_{II}$ is bounded above by
$|\gamma_k'(p_{II})|$. Then apply \eqref{odin}
and \eqref{dva} to \eqref{lalala} to obtain the desired result.
\end{proof}

\end{document}